\documentclass[11pt,leqno]{article}
\usepackage{amsmath,amssymb,amsthm,latexsym}
\usepackage{indentfirst}
\usepackage{amsmath}
\newtheorem{Definition}{\bf \large Definition}[section]
\newtheorem{theorem}{\bf \large Theorem}[section]
\newtheorem{PROPOSITION}{\bf \large Proposition}[section]
\newtheorem{corollary}{\bf \large Corollary}[section]

\newtheorem{remark}{\bf \large Remark}[section]
\newtheorem{Lemma}{\bf \large Lemma}[section]
\newtheorem{conjecture}{\bf \large Conjecture}[section]

\title{Local rigidity of constant mean curvature hypersurfaces  in space forms (II) }
\author {{ XinXin Cheng$^a$,~~, Yayun Chen$^b$, ~~Tongzhu Li$^c$,} \\
	\small{$^{a,c}$Department of Mathematics, Beijing Institute of
		Technology, Beijing, 100081, China.} \\
	\small{Beijing,100081,China.} \\
\small{$^b$School of Science, Chang'an University, Xi'an, 710064, China.}\\
	\small{ E-mail: 3120256206@bit.edu.cn,~~ cyy59@chd.edu.cn,~~ litz@bit.edu.cn.}}
\date{}

\begin{document}
	\maketitle
	\begin{abstract}
This is the second article of a sequence of research on the local rigidity of constant mean curvature (CMC) hypersurfaces in space forms. In the
previous one, we studied the local rigidity of CMC hypersurfaces whose the number of the distinct principal curvatures satisfies  $g\leq 3$.
In this paper, we study the local rigidity of CMC hypersurfaces with $g\geq 4$. When $g>4$,  we prove  that if the $k$-order mean curvatures $H_k$, $(k=2,\cdots, g-1)$  are constant and there exist  enough multiple principal curvatures, then the CMC hypersurface  is an isoparametric hypersurface. When  $g=4$, if   $H_2$ and $H_3$ are constant, then the CMC hypersurface  is an isoparametric hypersurface.
	\end{abstract}
	
	\medskip\noindent
	{\bf 2010 Mathematics Subject Classification:} 53B25, 53C24.
	\par\noindent {\bf Key words:} CMC hypersurface, minimal hypersurface, Bryant conjecture, isoparametric hypersurface.
	
	\section{Introduction}
	The study of the rigidity of CMC hypersurfaces   is a very interesting problem, especially for local rigidity of CMC hypersurface. This paper continues to study the local rigidity problem of CMC hypersurfaces based on paper \cite{chli}.
Up to now, the research on the rigidity of CMC hypersurfaces  has mainly focused on the study of global rigidity of CMC hypersurfaces.
	In the late 20th century, S. S. Chern has proposed the following Chern Conjecture about the global rigidity of minimal hypersurfaces,
	
	{\bf Chern Conjecture}:
	{\it Let $x:M^n\to \mathbb{S}^{n+1}$ be a closed  minimal hypersurfaces with constant scalar curvatures $R$ in $(n+1)$-dimensional sphere $\mathbb{S}^{n+1}$($n\geq 2$). Let $\mathbb{A}_R$ be the collection of all the possible values of such scalar curvature $R$, then $\mathbb{A}_R$ is a discrete subset of real numbers.}
	
	The Chern Conjecture remains open, but there are many partial results. Peng and Terng (\cite{peng},\cite{peng1}) made the first effort to solve the Chern Conjecture and confirmed the second gap of $\mathbb{A}_R$. Precisely, they proved that if the scalar curvature $R$ of $M^n$
	is a constant, then there exists a positive constant $C(n)$ depending only on $n$ such that if
	$n\leq S\leq n+C(n)$, then $S=n$. Later, the pinching constant $C(n)$ was improved to $\frac{n}{3},~n>3$
	by Cheng and Yang (\cite{yang2},\cite{yang3}), and to $\frac{3n}{7}$ by Suh and Yang (\cite{yang}), respectively.
	In 1993, Chang \cite{chang2} solved  Chern Conjecture for $n=3$.
	Over the years, there have been many important developments on the second gap  for  the closed minimal hypersurfaces
	in sphere, see for example \cite{cheng1},\cite{DGW},\cite{ge},\cite{M},\cite{weixu},\cite{xu4} and \cite{zhang}.
	
	For CMC hypersurfaces ($H\neq 0$), there are  also many global rigidities, especially  the squared norm of the second fundamental
	form $S$ has pinching phenomenon, which
	is  more complicated than the minimal hypersurface case. In \cite{li0}, \cite{li1}, \cite{li3}, H.Z. Li proved many important results about pinching theorem for CMC hypersurfaces,  including the proof of  Pinkall-Sterling conjecture.
	In \cite{xu3}, Xu proved the  pinching theorem for submanifolds with parallel mean curvature in a sphere. In \cite{xu2}, Xu and Tian generalized Suh-Yang's pinching theorem \cite{yang} to the case
	where $M^n$ is a compact hypersurface with constant scalar curvature and small constant
	mean curvature in $\mathbb{S}^{n+1}$.  As for the 3-dimensional hypersurface case,
	combining results of Almeida-Brito \cite{ad} and Chang \cite{chang1}, the CMC hypersurface with constant scalar curvature was classified.
	For  $n=4$, Tang-Yang in \cite{tali} proved that, if the scalar curvature $R\geq 0$,  $3$-mean curvature
	$H_3$ and the number $g$ of distinct principal curvatures  are constant, then $M^4$ is isoparametric. Tang-Wei-Yan in \cite{tang} and Tang-Yan in \cite{tang1} generalized the theorem of de Almeida
	and Brito (\cite{ad}) for $n=3$ to any dimension $n$.
	\begin{theorem}(\cite{tang,tang1})\label{tany}
		Let $M^n (n\geq 4)$ be a closed immersed hypersurface in $\mathbb{S}^{n+1}$. If the
		following conditions are satisfied:\\
		(i) the principal curvatures $\lambda_1,\lambda_2, \cdots, \lambda_n$ are distinct;\\
		(ii) $H_k=\sum_{i=1}^n\lambda_i^k, (k=1,\cdots,n-1)$ are constants;\\
		(iii) the scalar curvature $R\geq 0$;\\
		then $M^n$ is isoparametric and $R=0$.
	\end{theorem}
	The rigidity theorem mentioned above are all given under the assumption of global condition. In this paper, we   study the local rigidity of the CMC hypersurfaces $M^n$. Locally, the examples of CMC hypersurface are very abundant, and without appropriate conditions, CMC hypersurfaces have no any local rigidity, even in the case of closed CMC hypersurfaces. Let $A$ denote  the shape operator of the hypersurface $M^n$, then the $k-$order mean curvature $H_k$ is defined as $$H_k=\text{tr}(A^k).$$ Thus $H=\frac{1}{n}H_1=\frac{1}{n}tr(A).$ Here the local rigidity of CMC hypersurfaces  refers to the uniqueness of the CMC hypersurfaces under  appropriate condition on $k-$order curvatures.
	In \cite{chli}, we study the local rigidity of the CMC hypersurfaces, we obtained the following local rigidity theorem,
	\begin{theorem}\cite{chli}\label{chenli}
		Let $x: M^n\to \mathbb{M}^{n+1}(c), ~n\geq 4,$ be a piece of immersed CMC  hypersurface   in the $(n+1)$-dimensional
		space  form $\mathbb{M}^{n+1}(c)$. If the scalar curvature $R$ is constant and the number $g$ of distinct principal curvatures satisfies $g\leq 3$, then $M^n$ is an isoparametric hypersurface.
	\end{theorem}
	When $H=0$, Theorem \ref{chenli} solves the  high dimensional
	version of Bryant Conjecture.
	From his work on the exterior differential systems, R. Bryant  proposed the following Conjecture (see \cite{chang1}).
	
	{\bf Bryant Conjecture}: {\it A piece of minimal hypersurface of constant
		scalar curvature in $\mathbb{S}^4$ is isoparametric.}

	Based on the results of Theorem \ref{tany}, Theorem \ref{chenli} and Bryant Conjecture, we naturally have the following conjecture,
	\begin{conjecture}\label{conj}
		Let $M^n (n\geq 4)$ be  a piece of immersed CMC  hypersurface in $\mathbb{M}^{n+1}$ with $g$ distinct principal curvatures. If
		\begin{equation}\label{const}
			H_k~  is~ constant ~for~ k=2,\cdots, g-1,
		\end{equation}
		then $M^n$ is an isoparametric hypersurface.
	\end{conjecture}
	\begin{remark}
		In Conjecture \ref{conj}, the conditions (\ref{const}) are a system of algebraic equations, which imply that there exists only one independent principal curvature function among $g$ distinct principal curvatures. The conjecture essentially states that if the $g$ principal curvature functions satisfy the condition (\ref{const}), then the principal curvature functions must be constant.
	\end{remark}
	
	The goal of this paper is to try to prove Conjecture \ref{conj}.
	In \cite{M1}, R. Miyaoka has studied Chern's Conjecture in the Dupin case. A hypersurface is called Dupin if each principal curvature  is constant along its curvature direction. In \cite{M1} R. Miyaoka has proved that a closed proper Dupin hypersurface with constant mean curvature is isoparametric (i) if $g=3$, (ii) if
	$g=6$ and has constant scalar curvature, or (iii) if $g=4$ and has constant
	Lie curvature, and (iv) if $g=6$ and has constant Lie curvatures.
	This paper first proves the Conjecture \ref{conj} for the case of Dupin hypersurfaces,
	\begin{theorem}\label{th2}
		Let $x: M^n\to \mathbb{M}^{n+1}(c), (n\geq 4)$ be a piece of immersed CMC  hypersurface with $g$ distinct principal curvatures in the $(n+1)$-dimensional
		space  form $\mathbb{M}^{n+1}(c)$. If the following conditions are satisfies:\\
		(1) $H_k$  is constant for $k=2,\cdots, g-1$,\\
		(2) the hypersurface $M^n$ is a Dupin hypersurface,\\
		then $M^n$ is an isoparametric hypersurface.
	\end{theorem}
	
	A principal curvature is of  Dupin property, if the principal curvature  is constant along its curvature direction.  If a principal curvature of $M^n$ is multiple, then the principal curvature satisfies Dupin property. By Theorem \ref{th2} we can obtain the following corollary,
	\begin{corollary}\label{cor1}
		Let $x: M^n\to \mathbb{M}^{n+1}(c), (n\geq 4)$ be a piece of immersed CMC  hypersurface with $g$ distinct principal curvatures in the $(n+1)$-dimensional
		space  form $\mathbb{M}^{n+1}(c)$. If the following conditions are satisfies:\\
		(1) $H_k$  is constant for $k=2,\cdots, g-1$,\\
		(2) all principal curvatures are multiple,\\
		then $M^n$ is an isoparametric hypersurface.
	\end{corollary}
	
	\begin{remark}
		The global topological conditions have very strong limitations on Dupin hypersurfaces. Thorbergsson \cite{ce} showed that the number of distinct principal curvatures $g=1,2,3,4$ or $6$ for compact proper Dupin hypersurfaces embedded in $\mathbb{S}^{n+1}$. But locally examples of Dupin hypersurfaces  are very many.
		In fact, Pinkall \cite{ce} discovered the
		basic constructions of building tubes, cylinders, cones and surfaces of revolution over a Dupin
		hypersurface. It is important to note that these constructions may yield a compact proper
		Dupin hypersurface only if the original one is a sphere.
	\end{remark}
	When the number of distinct  simple principal curvatures is relatively small, we can prove Conjecture \ref{conj}.
	
	\begin{theorem}\label{th3}
		Let $x: M^n\to \mathbb{M}^{n+1}(c), (n\geq 4)$ be a piece of immersed CMC  hypersurface with $g (\geq 4)$ distinct principal curvatures in the $(n+1)$-dimensional
		space  form $\mathbb{M}^{n+1}(c)$. If the following conditions are satisfies:\\
		(1) $H_k$  is constant for $k=2,\cdots, g-1$,\\
		(2) there exists  at most one simple principal curvature,\\
		then $M^n$ is an isoparametric hypersurface.
	\end{theorem}
	
	When $g\leq 4$,  we  can also prove Conjecture \ref{conj},
	\begin{theorem}\label{th4}
		Let $x: M^n\to \mathbb{M}^{n+1}(c), (n\geq 5)$ be a piece of immersed CMC  hypersurface with $g$ distinct principal curvatures in the $(n+1)$-dimensional
		space  form $\mathbb{M}^{n+1}(c)$. If the following conditions are satisfies:\\
		(1) $H_2$ and $H_3$  are constant,\\
		(2) $g\leq 4$,\\
		then $M^n$ is an isoparametric hypersurface.
	\end{theorem}
	
	\begin{remark}
		In the proof process of Theorem \ref{th3} and \ref{th4}, our argument can essentially prove Conjecture \ref{conj} when there exists at least one multiple principal curvature.
		But when there are more many simple  principal curvatures,  we need to solve a very complex algebraic equation system. So this paper did not write about this complex solving process. I hope to find a concise way to handle it in the future.
	\end{remark}

	\par\noindent
	\section{Preliminaries}
	
	Let $\mathbb{M}^{n+1}(c)$ be an $(n+1)$-dimensional space form. When $c=1, \mathbb{M}^{n+1}(c)=\mathbb{S}^{n+1}$, which is an $(n+1)$-dimensional unit sphere;  When $c=0, \mathbb{M}^{n+1}(c)=\mathbb{R}^{n+1}$, which is an $(n+1)$-dimensional Euclidean space;  When $c=-1, \mathbb{M}^{n+1}(c)=\mathbb{H}^{n+1}$, which is an $(n+1)$-dimensional  hyperbolic space.
	
	Let $x:M^n\to \mathbb{M}^{n+1}(c)$ be an $n$-dimensional immersed  hypersurface
	in space form $\mathbb{M}^{n+1}(c)$. For any $p\in M^n$ we choose a local orthonormal frame $\{e_1,\cdots,e_n,e_{n+1}\}$
	around $p$ such that $e_1,\cdots,e_n$ are tangential to $M^n$ and $e_{n+1}$ is normal to $M^n$. Let $\{\omega_1,\cdots,\omega_n\}$ be the dual coframe and $\{\omega_{ij},~1\leq i,j\leq n\}$ be the connection $1$-forms.

In this paper we make the following convention
	on the range of indices, $$1\leq i,j,k\leq n.$$
	The structure equations of $M^n$ are given by
	\begin{equation*}
		\begin{split}
			&dw_{i}=\sum_jw_{ij}\wedge w_{j}, ~~ w_{ij}+w_{ji}=0,\\
			&dw_{ij}=\sum_k w_{ik}\wedge w_{kj}-\frac{1}{2}\sum_{k,l}R_{ijkl}w_{k}\wedge w_{l},
		\end{split}
	\end{equation*}
	where $R_{ijkl}$ is the curvature tensor of the induced metric on $M^n$.
	
	Let $II=\sum_{ij}h_{ij}\omega_i\otimes\omega_j$ denote the second fundamental form of the hypersurface, $ H=\frac{1}{n}\sum_{i}h_{ii}$ the mean curvature.
	The
	Gauss equation is
	\begin{equation}\label{gauss}
		R_{ijkl}=c(\delta_{ik}\delta_{jl}-\delta_{il}\delta_{jk})+h_{ik}h_{jl}-h_{il}h_{jk},
	\end{equation}
	and the Codazzi equation is
	\begin{equation}\label{coda}
		h_{ij,k}=h_{ik,j},
	\end{equation}
	where the covariant derivative of the second fundamental form is defined by
	$$\sum_mh_{ijm}w_{m}=dh_{ij}+\sum_mh_{mj}w_{mi}+\sum_mh_{im}w_{mj}.$$
	
	The second covariant derivative of the second fundamental form is defined by
	$$\sum_mh_{ij,km}w_{m}=dh_{ij,k}+\sum_mh_{mj,k}w_{mi}+\sum_mh_{im,k}w_{mj}+\sum_mh_{ij,m}\omega_{mk}.$$
	Thus we have the following Ricci identity
	\begin{equation}\label{ricd}
		h_{ijkl}-h_{ijlk}=\sum_mh_{mj}R_{mikl}+\sum_mh_{im}R_{mjkl}.
	\end{equation}
	
	By the Gauss equation (\ref{gauss}), we obtain the Ricci curvature $R_{ij}$ and the scalar curvature $R$ of the hypersurface,
	\begin{equation}\label{ricc}
		\begin{split}
			&R_{ij}=(n-1)c\delta_{ij}+nHh_{ij}-\sum_{m}h_{im}h_{mj},\\
			&R=n(n-1)c+n^2H^2-S,
		\end{split}
	\end{equation}
	where $S=\sum_{i,j}h_{ij}^2$ is the square norm of the second fundamental form.
	
	Let $A$ be the shape operator of the hypersurface $x$, which  is the dual  tensor  of the second fundamental form $II$. Define the $k$-order mean curvature $H_k$ by
	$$H_k = \text{Tr}(A^k),$$
	then $f_1=nH=nH_1,~f_2=\text{Tr}(A^2)=S$.
	
Let $\{\lambda_1,\cdots,\lambda_n\}$ be the eigenvalues of the  shape operator, which are called the principal curvatures of the hypersurface.
For $\lambda\in \{\lambda_1,\cdots,\lambda_n\}$, the curvature distribution $D(\lambda)$ is
	defined to be
	$$D_p(\lambda)=\{X\in T_pM^n|AX=\lambda X\},~~p\in M^n.$$
	\begin{Definition}
		Let $x:M^n\to \mathbb{M}^{n+1}(c)$ be an $n$-dimensional immersed  hypersurface
		in space form $\mathbb{M}^{n+1}(c)$. \\
		(1) If a principal curvature $\lambda_i$ satisfies $$X(\lambda_i)=0,~~ \forall~~ X\in D(\lambda_i),$$
		then the principal curvature $\lambda_i$ is called a Dupin principal curvature.\\
		(2) A hypersurface is called Dupin if each principal curvature  is a Dupin principal curvature.
	\end{Definition}
	When $dimD(\lambda_i)$ is constant $(=m)$ on $M^n$, $D(\lambda_i)$ is an involutive distribution. Moreover if $\lambda_i$ is a Dupin principal curvature,  the leaf $L$ is a piece of an $m$-dimensional subsphere of the
	curvature sphere  at each point $p\in M^n$.
	
	We denote by $g$  the number of the distinct principal curvatures. Then $g$ is a local constant. If $g$ is constant, then the principal curvatures are smooth. The results of Reckziegel \cite{reck} and Singley \cite{sin} imply  the following result (or see \cite{ce1}, p32.),
	\begin{PROPOSITION}(\cite{reck, sin})\label{rec}
		Let $x:M^n\to \mathbb{M}^{n+1}(c)$ be an $n$-dimensional immersed  hypersurface
		in space form $\mathbb{M}^{n+1}(c)$.  Then there always exists an open dense subset $U$ of $M^n$ on which the
		multiplicities of the principal curvatures are  constant, equivalently the number $g$ of distinct principal curvatures is constant.
	\end{PROPOSITION}

	If  the multiplicities of the principal curvatures are  constant on an open dense subset $U$ of $M^n$, then we can choose an orthonormal basis $\{e_1,e_2,\cdots,e_n\}$  for $TU$ such that
	$$(h_{ij})=diag(\lambda_1,\lambda_2,\cdots,\lambda_n).$$
	The smooth orthonormal frame $\{e_1,e_2,\cdots,e_n\}$ are called  unit principal vectors.

Therefore the $k$-order mean curvature
	$$H_k=\text{Tr}(A^k)=\lambda_1^k+\lambda_2^k+\cdots+\lambda_n^k.$$
By the Gauss equation,
$$R_{ijij}=c+\lambda_i\lambda_j.$$
	
	If the principal curvatures $\{\lambda_1,\cdots,\lambda_n\}$ are constant, then the hypersurface $x$ is called an isoparametric hypersurface.
	In $\mathbb{R}^{n+1}$ as well as $\mathbb{H}^{n+1}$, an isoparametric hypersurface has no more than two distinct principal curvatures and is completely
	classified (see \cite{ce}).
	\begin{theorem}\label{is0-1}(\cite{ce})
		Let $x:M^n\to \mathbb{R}^{n+1}$ be an isoparametric hypersurface, then $x$ is a hyperplane, or a hypersphere, or one of the following cylinders:
		$$x:\mathbb{R}^m\times \mathbb{S}^{n-m}(r)\to \mathbb{R}^{n+1},~~1\leq m\leq n-1.$$
	\end{theorem}
	\begin{theorem}\label{is0-2}(\cite{ce})
		Let $x:M^n\to \mathbb{H}^{n+1}$ be an isoparametric hypersurface, then $x$  is a hypersphere, or a horosphere, a hyperbolic hyperplane, or one of the following hyperbolic cylinders:
		$$x:\mathbb{H}^m(\sqrt{1+r^2})\times \mathbb{S}^{n-m}(r)\to \mathbb{H}^{n+1},~~1\leq m\leq n-1.$$
	\end{theorem}
	
	On the other hand, in $\mathbb{S}^{n+1}$, there are many more examples (see \cite{cat},\cite{ce}). M\"{u}nzner (\cite{mu})
	showed that the number $g$ of distinct principal curvatures of an isoparametric hypersurface in $\mathbb{S}^{n+1}$ must be $1, 2, 3, 4$ or $6$.
	Cartan \cite{cat} classified those with $g\leq 3$.
	\begin{theorem}\label{is0-3}(\cite{cat}\cite{ce})
		Let $x:M^n\to \mathbb{S}^{n+1}$ be an isoparametric hypersurface with $g(\leq 3)$ distinct principal curvatures, then\\
		(1) when $g=1$, $M^n$ is a great or small sphere in $\mathbb{S}^{n+1}$,\\
		(2) when $g=2$, $M^n$ is a Clifford torus
		$$\mathbb{S}^k(r)\times \mathbb{S}^{n-k}(\sqrt{1-r^2}),~~~1\leq k\leq n-1.$$
		(3) when $g=3$, $M^n$
		is a tube of constant radius over a standard Veronese embedding
		of a projective plane $FP^2$ into $\mathbb{S}^{3m+1}$,  where $F$ is the division algebra $R, \mathbb{C}$,
		$\mathbb{H}$(quaternions), $\mathbb{O}$(Cayley numbers) for $m = 1, 2, 4, 8,$ respectively.
	\end{theorem}
	For  $g=4$ and $g=6$, the isoparametric hypersurfaces were classified completely  by Cecil, Chi, Dorfmeister,
	Ferus, Karcher, Jensen, Miyaoka, M\"{u}nzner,   Neher, Takagi, Takahashi,  et al (see \cite{ce}).
	
	The Laplacian of the second fundamental form $h_{ij}$ is defined to be $\Delta h_{ij}=\sum_mh_{ij,mm}$, and by Ricci identity (\ref{ricd}),
	we have
	\begin{equation*}
		\begin{split}
			\frac{1}{2}\Delta S&=\frac{1}{2}\Delta (\sum_{i,j}h_{ij}^2)=\sum_{i,j,k}h_{ij,k}^2+\sum_{i,j,m}h_{ij}h_{ij,mm}\\
			&=\sum_{i,j,k}h_{ij,k}^2+\sum_{i,j,m}h_{ij}\Big(\sum_k h_{kkij}+\sum_{m,k} h_{mk}R_{mijk}+\sum_{m,k} h_{mi}R_{mkjk}\Big)\\
			&=\sum_{i,j,k}h_{ij,k}^2+S(nc-S)-cn^2H^2+nHf_3.
		\end{split}
	\end{equation*}
	Therefore
	\begin{equation}\label{lap-1}
		\frac{1}{2}\Delta S=\sum_{i,j,k}h_{ij,k}^2+(nc-S)S-cn^2H^2+nHf_3.
	\end{equation}
	Similarly,
	\begin{equation}\label{lap-2}
		\frac{1}{3}\Delta f_3=2\sum_{i,j,k,m}h_{im,k}h_{jm,k}h_{ij}+(nc-S)f_3-ncHS+nHf_4.
	\end{equation}
	
	\section{Proof of Main Theorems}
	Let $x: M^n\to \mathbb{M}^{n+1}(c)$ be a piece of immersed CMC  hypersurface with $g$ distinct principal curvatures  in the $(n+1)$-dimensional
	space  form $\mathbb{M}^{n+1}(c)$. Since we study the local  properties of the hypersurface, by Proposition \ref{rec}
	we can assume that $g$ is constant on $M^n$. Let $\bar{\lambda}_1,\cdots,\bar{\lambda}_g$ are $g$ distinct principal curvatures with multiplicities $m_1,\cdots,m_g$ respectively.
	Thus we can choose a orthonormal basis $\{e_1,e_2,\cdots,e_n\}$  for $TM^n$ with respect to the first fundamental form $I=dx\cdot dx$,
	consisting of unit principal vectors. Thus under the basis, the second fundamental form satisfies
	\begin{equation}\label{basis-1'}
		(h_{ij})=diag(\lambda_1,\lambda_2,\cdots,\lambda_n)=diag(\underbrace{\bar{\lambda}_1,\cdots,\bar{\lambda}_1}_{m_1},
		\underbrace{\bar{\lambda}_2,\cdots,\bar{\lambda}_2}_{m_2},
		\cdots,\underbrace{\bar{\lambda}_g,\cdots,\bar{\lambda}_g}_{m_g}).
	\end{equation}

Under the unit principal vectors $\{e_1,e_2,\cdots,e_n\}$,  we define the following index set,
	$$[\bar{\lambda}_i]=\{k\in \{1,2,\cdots,n\}| ~\lambda_k=\bar{\lambda}_i\},~~i=1,2,\cdots, g.$$

By the definition of the covariant derivative of the second fundamental form
	$\sum_mh_{ijm}w_{m}=dh_{ij}+\sum_mh_{mj}w_{mi}+\sum_mh_{im}w_{mj}$ and (\ref{basis-1'}),
	we obtain the following equations,
	\begin{equation}\label{con2'}
		\begin{split}
			&e_i(\lambda_j)=e_i(h_{jj})=h_{jj,i},\\
			&(\lambda_i-\lambda_j)\omega_{ij}=\sum_mh_{ijm}w_{m},~~i\neq j.
		\end{split}
	\end{equation}
	\begin{Lemma}\label{le1}
		If $m_i\geq 2$, then $e_j(\bar{\lambda}_i)=0,~~j\in[\bar{\lambda}_i].$
	\end{Lemma}
	\begin{proof}
		We assume that $m_1\geq 2$. There are two
		indices $i,j\in [\bar{\lambda}_1]$ and $i\neq j$. By the second equation in (\ref{con2'}),
		we have $h_{ij,i}=0$. Since $h_{ij,i}=h_{ii,j}=e_j(h_{ii})=e_j(\bar{\lambda}_1)$, thus $e_j(\bar{\lambda}_1)=0,~~j\in[\bar{\lambda}_1]$.
	\end{proof}

For $\lambda_i\neq \lambda_j$, we have the following equation by (\ref{con2'}),
	\begin{equation}\label{con2-1}
		\omega_{ij}=\sum_m\frac{h_{ij,m}}{\lambda_i-\lambda_j}w_{m},~~\lambda_i\neq \lambda_j.
	\end{equation}
	By the second covariant derivative of the second fundamental form and (\ref{con2-1}), we obtain the following equations,
	\begin{equation}\label{con2}
		\begin{split}
			&h_{ij,ij}=e_j(h_{ij,i})+2\sum_{m\notin [\lambda_i]}\frac{h_{ij,m}^2}{\lambda_m-\lambda_i}+\sum_{m\notin [\lambda_j]}\frac{h_{ii,m}h_{jj,m}}{\lambda_m-\lambda_j},\\
			&h_{ij,ji}=e_i(h_{ij,j})+2\sum_{m\notin [\lambda_j]}\frac{h_{ij,m}^2}{\lambda_m-\lambda_j}+\sum_{m\notin [\lambda_i]}\frac{h_{ii,m}h_{jj,m}}{\lambda_m-\lambda_i}.
		\end{split}
	\end{equation}
	By Ricci identity,
	\begin{equation*}
		\begin{split}
			&(\lambda_i-\lambda_j)R_{ijij}=e_j(h_{ij,i})-e_i(h_{ij,j})
			+2\sum_{m\notin [\lambda_i],m\notin [\lambda_j]}\frac{h_{ij,m}^2(\lambda_i-\lambda_j)}{(\lambda_m-\lambda_i)(\lambda_m-\lambda_j)}\\
			&+\sum_{m\notin [\lambda_i],m\notin [\lambda_j]}\frac{h_{ii,m}h_{jj,m}(\lambda_j-\lambda_i)}{(\lambda_m-\lambda_i)(\lambda_m-\lambda_j)}
			+\frac{h_{ii,i}h_{jj,i}+h_{jj,j}h_{ii,j}-2(h_{ii,j}^2+h_{jj,i}^2)}{\lambda_i-\lambda_j}.
		\end{split}
	\end{equation*}
	i.e.,
	\begin{equation}\label{con3-1}
		\begin{split}
			R_{ijij}&=\frac{e_j(h_{ij,i})-e_i(h_{ij,j})}{\lambda_i-\lambda_j}
			+\sum_{m\notin [\lambda_i],m\notin [\lambda_j]}\frac{2h_{ij,m}^2-h_{ii,m}h_{jj,m}}{(\lambda_m-\lambda_i)(\lambda_m-\lambda_j)}\\
			&+\frac{h_{ii,i}h_{jj,i}+h_{jj,j}h_{ii,j}-2(h_{ii,j}^2+h_{jj,i}^2)}{(\lambda_i-\lambda_j)^2},~~\lambda_i\neq \lambda_j.
		\end{split}
	\end{equation}

Since the $k$-order mean curvatures are constant, then,
	\begin{equation}\label{curva-1}
		\begin{split}
			&m_1\bar{\lambda}_1+m_2\bar{\lambda}_2+\cdots+m_g\bar{\lambda}_g=c_1,\\
			&m_1\bar{\lambda}_1^2+m_2\bar{\lambda}_2^2+\cdots+m_g\bar{\lambda}_g^2=c_2,\\
			&\cdots\cdots\cdots\cdots\cdots\cdots\\
			&m_1\bar{\lambda}_1^{r}+m_2\bar{\lambda}_2^{r}+\cdots+m_g\bar{\lambda}_g^{r}=c_{r},
		\end{split}
	\end{equation}
	where $c_1,c_2,\cdots,c_{r}$ are constant.
	
	By the equation system (\ref{curva-1}), if $r\geq g$,  then  it is easy to prove that the principal curvatures $\bar{\lambda}_1,\cdots,\bar{\lambda}_g$ are constant. Thus we have the following results,
	\begin{PROPOSITION}\label{pro1}
		Let $x: M^n\to \mathbb{M}^{n+1}(c),$ be a piece of immersed CMC  hypersurface with $g$ distinct principal curvatures in the $(n+1)$-dimensional
		space  form $\mathbb{M}^{n+1}(c)$. If the following conditions are satisfies:\\
		(1) $H_k~ (k=2,\cdots, r)$ are constant,\\
		(2)  $r\geq g$,\\
		then $M^n$ is an isoparametric hypersurface.
	\end{PROPOSITION}
	
	Now we assume that $r=g-1$, i.e.,
	\begin{equation}\label{cond1}
		\begin{split}
			&m_1\bar{\lambda}_1+m_2\bar{\lambda}_2+\cdots+m_g\bar{\lambda}_g=c_1,\\
			&m_1\bar{\lambda}_1^2+m_2\bar{\lambda}_2^2+\cdots+m_g\bar{\lambda}_g^2=c_2,\\
			&\cdots\cdots\cdots\cdots\cdots\cdots\\
			&m_1\bar{\lambda}_1^{g-1}+m_2\bar{\lambda}_2^{g-1}+\cdots+m_g\bar{\lambda}_g^{g-1}=c_{g-1},\\
			&m_1\bar{\lambda}_1^{g}+m_2\bar{\lambda}_2^{g}+\cdots+m_g\bar{\lambda}_g^{g}=H_g.
		\end{split}
	\end{equation}
	where $c_1,c_2,\cdots,c_{g-1}$ are constant.  Differentiating the equations in (\ref{cond1}) by $e_k$  for each $k=1,2,\cdots, n$, we get
	\begin{equation}\label{cond2}
		\left(\begin{array}{cccc}
			1&1 &\cdots&1\\
			\bar{\lambda}_1&\bar{\lambda}_2 &\cdots&\bar{\lambda}_g\\
			\vdots&\vdots&\ddots&\vdots\\
			\bar{\lambda}_1^{g-2}&\bar{\lambda}_2^{g-2}&\cdots&\bar{\lambda}_g^{g-2}\\
			\bar{\lambda}_1^{g-1}&\bar{\lambda}_2^{g-1}&\cdots&\bar{\lambda}_g^{g-1}\\
		\end{array}\right)
		\left(\begin{array}{c}
			m_1e_k(\bar{\lambda}_1)\\
			m_2e_k(\bar{\lambda}_2)\\
			\vdots\\
			m_{g-1}e_k(\bar{\lambda}_{g-1})\\
			m_ge_k(\bar{\lambda}_{g})\\
		\end{array}\right)=
		\left(\begin{array}{c}
			0\\
			0\\
			\vdots\\
			0\\
			\frac{e_k(H_g)}{g}\\
		\end{array}\right).
	\end{equation}
	From (\ref{cond2}), we can obtain
	\begin{equation}\label{cond3}
		m_je_k(\bar{\lambda}_j)=\frac{(-1)^{g+1}}{g}\frac{e_k(H_g)}{\prod_{l=1;l\neq j}^g(\bar{\lambda}_{l}-\bar{\lambda}_{j})},
	\end{equation}
	and
	\begin{equation}\label{cond4}
		m_je_k(\bar{\lambda}_j)=\frac{\prod_{l=1;l\neq i}^g(\bar{\lambda}_{l}-\bar{\lambda}_{i})}{\prod_{l=1;l\neq j}^g(\bar{\lambda}_{l}-\bar{\lambda}_{j})}m_ie_k(\bar{\lambda}_i),~~1\leq i,j\leq g,~~1\leq k\leq n.
	\end{equation}

	Let $P(\bar{\lambda}_i)=m_i\prod_{l=1;l\neq i}^g(\bar{\lambda}_l-\bar{\lambda}_i)$, then (\ref{cond4}) can write as
	\begin{equation}\label{con41}
		P(\bar{\lambda}_j)e_k(\bar{\lambda}_j)=P(\bar{\lambda}_i)e_k(\bar{\lambda}_i), ~~1\leq i,j\leq g,~~1\leq k\leq n.
	\end{equation}

	Differentiating the equations in (\ref{con41}) by $e_k$  for each $k=1,2,\cdots, n$, we obtain
\begin{equation}\label{cond5}
		\begin{split}
e_k[e_k(\bar{\lambda}_j)]&=e_k(\bar{\lambda}_i)^2\Big[\sum_{l\neq i}\frac{P(\bar{\lambda}_i)-P(\bar{\lambda}_l)}{P(\bar{\lambda}_l)(\bar{\lambda}_l-\bar{\lambda}_i)}-\sum_{l\neq j}\frac{P(\bar{\lambda}_i)(P(\bar{\lambda}_j)-P(\bar{\lambda}_l))}{P(\bar{\lambda}_l)P(\bar{\lambda}_j)(\bar{\lambda}_l
-\bar{\lambda}_j)}\Big]\frac{P(\bar{\lambda}_i)}{P(\bar{\lambda}_j)}\\
			&+\frac{P(\bar{\lambda}_i)}{P(\bar{\lambda}_j)}e_k[e_k(\bar{\lambda}_i)],~~1\leq i,j\leq g,~~1\leq k\leq n.
		\end{split}
	\end{equation}

{\bf Proof of Theorem \ref{th2}}. The hypersurface is Dupin hypersurface if and only if
	$$e_j(\bar{\lambda}_i)=0,~~j\in[\bar{\lambda}_i],~~1\leq i\leq g.$$
	If the hypersurface $M^n$ is Dupin, combining the equation (\ref{cond4}), we get
	$$e_j(\bar{\lambda}_i)=0,~,~~1\leq i,j \leq g.$$
	Thus we finish the proof of Theorem \ref{th2}.

	Lemma \ref{le1} implies that if a principal curvature $\bar{\lambda}_i$ is multiple, then the principal curvature is a Dupin principal curvature. Thus Corollary \ref{cor1} is proved by Theorem \ref{th2}.
	
{\bf Proof of Theorem \ref{th3}}.  By the results in \cite{chli}, We can assume that the number of the distinct principal curvatures satisfies $g\geq 4$. Due to the result of Corollary \ref{cor1}, to prove Theorem \ref{th3}, we need to prove the case that there exist  $g-1$ principal curvatures  that are Dupin principal curvatures. Thus
		we can assume that under the orthonormal basis $\{e_1,e_2,\cdots,e_n\}$,
		$$(h_{ij})=diag(\lambda_1,\lambda_2,\cdots,\lambda_n)=diag(\bar{\lambda}_1,
		\underbrace{\bar{\lambda}_2,\cdots,\bar{\lambda}_2}_{m_2},
		\cdots,\underbrace{\bar{\lambda}_g,\cdots,\bar{\lambda}_g}_{m_g}),$$
		
		By Lemma \ref{le1} and the equation (\ref{cond4}),  we get
		\begin{equation}\label{pro1-1}
			e_j(\bar{\lambda}_1)=e_j(\bar{\lambda}_2)=\cdots=e_j(\bar{\lambda}_g)=0,~~2\leq j\leq n.
		\end{equation}
		
		We need to prove that
		\begin{equation}\label{pro1-11}
			e_1(\bar{\lambda}_1)=\cdots=e_1(\bar{\lambda}_g)=0.
		\end{equation}
		
		Next  we prove the equation (\ref{pro1-11}) by contradiction. We assume that $$e_1(\bar{\lambda}_1)\neq 0,\cdots,e_1(\bar{\lambda}_g)\neq 0.$$
		Now fix index $i,j>1$ such that $\lambda_i\neq\lambda_j$, using the second covariant derivative of the second fundamental form and (\ref{con2-1}), we obtain the following equations,
		\begin{equation}\label{pro1-2}
			\begin{split}
				&h_{11,ij}=2\sum_{m\neq 1}\frac{h_{m1,i}h_{m1,j}}{\lambda_m-\lambda_1}+\frac{h_{11,1}h_{1i,j}}{\lambda_1-\lambda_i},\\
				&h_{11,ji}=2\sum_{m\neq 1}\frac{h_{m1,i}h_{m1,j}}{\lambda_m-\lambda_1}+\frac{h_{11,1}h_{1i,j}}{\lambda_1-\lambda_j}.
			\end{split}
		\end{equation}
		Combining Ricci identity and (\ref{pro1-2}), we obtain
		\begin{equation*}
			\frac{h_{11,1}h_{1i,j}}{(\lambda_1-\lambda_i)(\lambda_1-\lambda_j)}=0.
		\end{equation*}
		Since $h_{11,1}=e_1(\lambda_1)\neq 0$, so
		\begin{equation}\label{pro1-4}
			h_{1i,j}=0,~~i,j>1, \lambda_i\neq\lambda_j.
		\end{equation}
		By (\ref{con3-1}) and (\ref{con41}), we have
		\begin{equation}\label{con3-11}
			\begin{split}
				R_{1i1i}&=\frac{e_1(h_{ii,1})}{\lambda_i-\lambda_1}+\frac{h_{11,1}h_{ii,1}-2h_{ii,1}^2}{(\lambda_i-\lambda_1)^2}\\
				&=\frac{e_1e_1(\bar{\lambda}_i)}{\lambda_i-\bar{\lambda}_1}
				 +\frac{P(\bar{\lambda}_1)P(\bar\lambda_i)-2P(\bar{\lambda}_1)^2}{P(\bar\lambda_i)^2(\lambda_i-\bar{\lambda}_1)^2}e_1(\bar{\lambda}_1)^2,~~i>1.
			\end{split}
		\end{equation}
		
		Since $g\geq 4$, there exist at least three multiple principal curvatures $\bar{\lambda}_2, \bar{\lambda}_3,\bar{ \lambda}_4$. Fixed index
		$i_1\in [\bar{\lambda}_2],~i_2\in [\bar{\lambda}_3],~i_3\in [\bar{\lambda}_4]$, by (\ref{con3-11}),
		\begin{equation}\label{con3-12}
			\begin{split}
				&e_1e_1(\bar{\lambda}_2)=(\bar{\lambda}_2-\bar{\lambda}_1)R_{1i_11i_1}
				 +\frac{P(\bar{\lambda}_1)P(\bar{\lambda}_2)-2P(\bar{\lambda}_1)^2}{P(\bar{\lambda}_2)^2(\bar{\lambda}_1-\bar{\lambda}_2)}e_1(\bar{\lambda}_1)^2;\\
				&e_1e_1(\bar{\lambda}_3)=(\bar{\lambda}_3-\bar{\lambda}_1)R_{1i_21i_2}
				 +\frac{P(\bar{\lambda}_1)P(\bar{\lambda}_3)-2P(\bar{\lambda}_1)^2}{P(\bar{\lambda}_3)^2(\bar{\lambda}_1-\bar{\lambda}_3)}e_1(\bar{\lambda}_1)^2;\\
				&e_1e_1(\bar{\lambda}_4)=(\bar{\lambda}_4-\bar{\lambda}_1)R_{1i_31i_3}
				 +\frac{P(\bar{\lambda}_1)P(\bar{\lambda}_4)-2P(\bar{\lambda}_1)^2}{P(\bar{\lambda}_4)^2(\bar{\lambda}_1-\bar{\lambda}_4)}e_1(\bar{\lambda}_1)^2.
			\end{split}
		\end{equation}
		By (\ref{cond5}), we have the following equations,
		\begin{equation}\label{con3-13}
			\begin{split}
				&P(\bar{\lambda}_i)e_1e_1(\bar{\lambda}_i)=P(\bar{\lambda}_j)e_1e_1(\bar{\lambda}_j)\\
				&-\frac{P(\bar{\lambda}_j)P(\bar{\lambda}_1)^2e_1(\bar{\lambda}_1)^2}{P(\bar{\lambda}_i)}\Big[\sum_{m\neq j}\frac{P(\bar{\lambda}_m)-P(\bar{\lambda}_j)}{P(\bar{\lambda}_m)P(\bar{\lambda}_j)(\bar{\lambda}_m-\bar{\lambda}_j)}
				-\sum_{m\neq i}\frac{P(\bar{\lambda}_m)-P(\bar{\lambda}_i)}{P(\bar{\lambda}_m)P(\bar{\lambda}_i)(\bar{\lambda}_m-\bar{\lambda}_i)}\Big].
			\end{split}
		\end{equation}
		Combining (\ref{con3-12}) and (\ref{con3-13}), we get
		\begin{equation}\label{con3-14}
			\begin{split}
				&P(\bar{\lambda}_2)(\bar{\lambda}_2-\bar{\lambda}_1)R_{1i_11i_1}+
				 \frac{P(\bar{\lambda}_1)P(\bar{\lambda}_2)-2P(\bar{\lambda}_1)^2}{P(\bar{\lambda}_2)(\bar{\lambda}_1-\bar{\lambda}_2)}e_1(\bar{\lambda}_1)^2\\
				&=P(\bar{\lambda}_3)(\bar{\lambda}_3-\bar{\lambda}_1)R_{1i_21i_2}+
				 \frac{P(\bar{\lambda}_1)P(\bar{\lambda}_3)-2P(\bar{\lambda}_1)^2}{P(\bar{\lambda}_3)(\bar{\lambda}_1-\bar{\lambda}_3)}e_1(\bar{\lambda}_1)^2\\
				&-\frac{P(\bar{\lambda}_3)P(\bar{\lambda}_1)^2e_1(\bar{\lambda}_1)^2}{P(\bar{\lambda}_2)}\Big[\sum_{m\neq 3}\frac{P(\bar{\lambda}_m)-P(\bar{\lambda}_3)}{P(\bar{\lambda}_m)P(\bar{\lambda}_3)(\bar{\lambda}_m-\bar{\lambda}_3)}
				-\sum_{m\neq 2}\frac{P(\bar{\lambda}_m)-P(\bar{\lambda}_2)}{P(\bar{\lambda}_m)P(\bar{\lambda}_2)(\bar{\lambda}_m-\bar{\lambda}_2)}\Big],\\
			\end{split}
		\end{equation}
		and
		\begin{equation}\label{con3-15}
			\begin{split}
				&P(\bar{\lambda}_2)(\bar{\lambda}_2-\bar{\lambda}_1)R_{1i_11i_1}+
				 \frac{P(\bar{\lambda}_1)P(\bar{\lambda}_2)-2P(\bar{\lambda}_1)^2}{P(\bar{\lambda}_2)(\bar{\lambda}_1-\bar{\lambda}_2)}e_1(\bar{\lambda}_1)^2\\
				&=P(\bar{\lambda}_4)(\bar{\lambda}_4-\bar{\lambda}_1)R_{1i_31i_3}+
				 \frac{P(\bar{\lambda}_1)P(\bar{\lambda}_4)-2P(\bar{\lambda}_1)^2}{P(\bar{\lambda}_4)(\bar{\lambda}_1-\bar{\lambda}_4)}e_1(\bar{\lambda}_1)^2\\
				&-\frac{P(\bar{\lambda}_4)P(\bar{\lambda}_1)^2e_1(\bar{\lambda}_1)^2}{P(\bar{\lambda}_2)}\Big[\sum_{m\neq 4}\frac{P(\bar{\lambda}_m)-P(\bar{\lambda}_4)}{P(\bar{\lambda}_m)P(\bar{\lambda}_4)(\bar{\lambda}_m-\bar{\lambda}_4)}
				-\sum_{m\neq 2}\frac{P(\bar{\lambda}_m)-P(\bar{\lambda}_2)}{P(\bar{\lambda}_m)P(\bar{\lambda}_2)(\bar{\lambda}_m-\bar{\lambda}_2)}\Big].
			\end{split}
		\end{equation}
		Let
		\begin{equation*}
			\begin{split}
				Q(\bar{\lambda}_i,\bar{\lambda}_j)
				&=\frac{P(\bar{\lambda}_i)-2P(\bar{\lambda}_1)}{P(\bar{\lambda}_1)P(\bar{\lambda}_i)(\bar{\lambda}_1-\bar{\lambda}_i)}
				-\frac{P(\bar{\lambda}_j)-2P(\bar{\lambda}_1)}{P(\bar{\lambda}_1)P(\bar{\lambda}_j)(\bar{\lambda}_1-\bar{\lambda}_j)}\\
				&-\frac{P(\bar{\lambda}_j)}{P(\bar{\lambda}_i)}(\sum_{m\neq i}\frac{P(\bar{\lambda}_m)-P(\bar{\lambda}_i)}{P(\bar{\lambda}_m)P(\bar{\lambda}_i)(\bar{\lambda}_m-\bar{\lambda}_i)}
				-\sum_{m\neq j}\frac{P(\bar{\lambda}_m)-P(\bar{\lambda}_j)}{P(\bar{\lambda}_m)P(\bar{\lambda}_j)(\bar{\lambda}_m-\bar{\lambda}_j)}).
			\end{split}
		\end{equation*}
		By (\ref{con3-14}) and (\ref{con3-15}), we have
		 $$Q(\bar{\lambda}_2,\bar{\lambda}_3)P(\bar{\lambda}_1)^2e_1(\bar{\lambda}_1)^2=P(\bar{\lambda}_3)(\bar{\lambda}_3-\bar{\lambda}_1)R_{1i_21i_2}
		-P(\bar{\lambda}_2)(\bar{\lambda}_2-\bar{\lambda}_1)R_{1i_11i_1},$$
		 $$Q(\bar{\lambda}_2,\bar{\lambda}_4)P(\bar{\lambda}_1)^2e_1(\bar{\lambda}_1)^2=P(\bar{\lambda}_4)(\bar{\lambda}_4-\bar{\lambda}_1)R_{1i_31i_3}
		-P(\bar{\lambda}_2)(\bar{\lambda}_2-\bar{\lambda}_1)R_{1i_11i_1},$$
		which implies that
		\begin{equation}\label{con3-16}
			\begin{split}
				&Q(\bar{\lambda}_2,\bar{\lambda}_3)\Big[P(\bar{\lambda}_4)(\bar{\lambda}_4-\bar{\lambda}_1)R_{1i_31i_3}
				-P(\bar{\lambda}_2)(\bar{\lambda}_2-\bar{\lambda}_1)R_{1i_11i_1}\Big]\\
				&=Q(\bar{\lambda}_2,\bar{\lambda}_4)\Big[P(\bar{\lambda}_3)(\bar{\lambda}_3-\bar{\lambda}_1)R_{1i_21i_2}
				-P(\bar{\lambda}_2)(\bar{\lambda}_2-\bar{\lambda}_1)R_{1i_11i_1}\Big].
			\end{split}
		\end{equation}
		
		we obtain the  equation system (\ref{curva-1}) and (\ref{con3-16}) of $g$ algebraic equations about the $g$ distinct principal curvature $\bar{\lambda}_1,\bar{\lambda}_2,\cdots, \bar{\lambda}_g$. Because the equations in the system [\ref{curva-1}] are symmetric functions of different orders, thus these equations are independent with each other. On the other hands, since the equation (\ref{con3-16}) about variables $\{\lambda_1, \cdots, \lambda_n\}$ is an asymmetric algebraic equation,  the equation (\ref{con3-16})  can not  be generated by the equation system [\ref{curva-1}]. Thus from the  equation system [(\ref{curva-1})-(\ref{con3-16})], we know that  the $g$ distinct principal curvature $\bar{\lambda}_1,\bar{\lambda}_2,\cdots, \bar{\lambda}_g$ are constant. So the hypersurface $M^n$ is an isoparametric hypersurface.
Combining  Lemma \ref{le1}, we finish the proof of Theorem \ref{th3}.

\textbf{Proof of Theorem\ref{th4}.}
		By the results in \cite{chli}, we need to prove $g=4$. We can assume that under the orthonormal basis $\{e_1,e_2,\cdots,e_n\}$,
		$$(h_{ij})=diag(\lambda_1,\lambda_2,\cdots,\lambda_n)=diag(\underbrace{\bar{\lambda}_1,\cdots, \bar{\lambda}_1}_{m_1},
		\underbrace{ \bar{\lambda}_2,\cdots,\bar{\lambda}_2}_{m_2},
		\cdots,\underbrace{\bar{\lambda}_4,\cdots,\bar{\lambda}_4}_{m_4}).$$
		We divide the proof of Theorem\ref{th4} into four cases:\\
		(i) {\bf Case 1}, $m_1\geq 2, m_2\geq 2, m_3\geq 2, m_4\geq 2$,\\
		(ii) {\bf Case 2}, $m_1=1, m_2\geq 2, m_3\geq 2, m_4\geq 2$,\\
		(iii) {\bf Case 3}, $m_1=m_2=1, m_3\geq 2, m_4\geq 2$,\\
		(iv) {\bf Case 4}, $m_1=m_2=m_3=1,  m_4\geq 2$,\\
		
		{\bf Case 1} is proved by Corollary \ref{cor1}.
		
		{\bf Case 2} is proved by Theorem \ref{th3}.
		
		Now we prove {\bf Case 3}.
		$$(h_{ij})=diag(\lambda_1,\lambda_2,\cdots,\lambda_n)=diag(\bar{\lambda}_1, \bar{\lambda}_2,
		\underbrace{ \bar{\lambda}_3,\cdots,\bar{\lambda}_3}_{m_3},\underbrace{\bar{\lambda}_4,\cdots,\bar{\lambda}_4}_{m_4}).$$
		Since $m_3\geq 2,,m_4\geq 2$, by Lemma \ref{le1} and the equation (\ref{cond4}),  we get
		\begin{equation}\label{pro1-1}
			e_j(\bar{\lambda}_1)=e_j(\bar{\lambda}_2)=\cdots=e_j(\bar{\lambda}_4)=0,~~3\leq j\leq n.
		\end{equation}
		To prove Theorem\ref{th4}, we need to prove that
		$$e_1(\bar{\lambda}_1)=\cdots=e_1(\bar{\lambda}_4)=0,~~e_2(\bar{\lambda}_1)=\cdots=e_2(\bar{\lambda}_4)=0.$$
		Next we prove it by contradiction. We assume that $$e_1(\bar{\lambda}_1)\neq 0,\cdots,e_1(\bar{\lambda}_g)\neq 0,~~e_2(\bar{\lambda}_1)\neq0,\cdots,e_2(\bar{\lambda}_g)\neq 0.$$
		
		We fix the index $i\in [\bar{\lambda}_3], j\in [\bar{\lambda}_4]$,
		by the second covariant derivative of the second fundamental form and (\ref{con2-1}), we obtain the following equations,
		\begin{equation*}
			\begin{split}
				&h_{11,ij}=2\sum_{m\neq 1}\frac{h_{m1,i}h_{m1,j}}{\lambda_m-\lambda_1}+\frac{h_{11,1}h_{1i,j}}{\lambda_1-\lambda_i}+\frac{h_{11,2}h_{2i,j}}{\lambda_2-\lambda_i},\\
				&h_{11,ji}=2\sum_{m\neq 1}\frac{h_{m1,i}h_{m1,j}}{\lambda_m-\lambda_1}+\frac{h_{11,1}h_{1i,j}}{\lambda_1-\lambda_j}+\frac{h_{11,2}h_{2i,j}}{\lambda_2-\lambda_j}.
			\end{split}
		\end{equation*}
		By Ricci identity, we get
		\begin{equation}\label{pro2-1}
			 \frac{h_{11,1}h_{1i,j}}{(\lambda_1-\lambda_i)(\lambda_1-\lambda_j)}+\frac{h_{11,2}h_{2i,j}}{(\lambda_2-\lambda_i)(\lambda_2-\lambda_j)}=0.
		\end{equation}
		Since $h_{ij,j}=h_{ij,i}=0$, combining  (\ref{con3-1}), we obtain
		\begin{equation}\label{pro2-2}
			\begin{split}
				R_{ijij}&=\frac{2h_{1i,j}^2}{(\bar{\lambda}_1-\bar{\lambda}_3)(\bar{\lambda}_1-\bar{\lambda}_4)}+
				\frac{2h_{2i,j}^2}{(\bar{\lambda}_2-\bar{\lambda}_3)(\bar{\lambda}_2-
					\bar{\lambda}_4)}\\
				&-\frac{h_{ii,1}h_{jj,1}}{(\bar{\lambda}_1-\bar{\lambda}_3)(\bar{\lambda}_1-\bar{\lambda}_4)}
				-\frac{h_{ii,2}h_{jj,2}}{(\bar{\lambda}_2-\bar{\lambda}_3)(\bar{\lambda}_2-\bar{\lambda}_4)}.
			\end{split}
		\end{equation}
		We consider the matrix $G=(h_{1i,j}),~~i\in [\bar{\lambda}_3], j\in [\bar{\lambda}_4]$ .Since $m_3\geq 2, m_4\geq 2$, using matrix' singular value decomposition, we can choose a basis  in $D(\bar{\lambda}_3)$  and a
		basis in $D(\bar{\lambda}_4)$ such that
		\begin{equation}\label{ma}
			G=(h_{1i,j})=\left(\begin{matrix}\Lambda&0\\
				0&0\end{matrix}\right),
		\end{equation}
		where $\Lambda=diag(c_1,\cdots,c_t)$ is a diagonal matrix.
		
		Since $h_{11,1}h_{11,2}=e_1(\lambda_1)e_2(\lambda_1)\neq 0$, from (\ref{pro2-1}), we obtain that if $h_{1i,j}=0$, the $h_{2i,j}=0$.
		From (\ref{ma}), there exist index $i_0\in [\bar{\lambda}_3], j_0\in [\bar{\lambda}_4]$ such that $h_{1i_0,j_0}=0$, and $h_{2i_0,j_0}=0$.
		Thus combining equation (\ref{pro2-2}), using $R_{ijij}=R_{i_0j_0i_0j_0}$, we have
		\begin{equation}\label{pro2-3}
			\frac{h_{1i,j}^2}{(\bar{\lambda}_1-\bar{\lambda}_3)(\bar{\lambda}_1-\bar{\lambda}_4)}+
			\frac{h_{2i,j}^2}{(\bar{\lambda}_2-\bar{\lambda}_3)(\bar{\lambda}_2-
				\bar{\lambda}_4)}=0,~~\forall~~i\in [\bar{\lambda}_3], j\in [\bar{\lambda}_4].
		\end{equation}
		Hence we obtain $h_{1ij}=h_{2ij}=0$.
		
		From (\ref{pro2-2}), combining (\ref{con41}) we can get
		\begin{equation}\label{pro2-12}
			\begin{split}
				R_{ijij}&=-\frac{h_{ii,1}h_{jj,1}}{(\bar{\lambda}_1-\bar{\lambda}_3)(\bar{\lambda}_1-\bar{\lambda}_4)}
				-\frac{h_{ii,2}h_{jj,2}}{(\bar{\lambda}_2-\bar{\lambda}_3)(\bar{\lambda}_2-\bar{\lambda}_4)}\\
				&=\frac{-P(\bar{\lambda}_1)^2}{P(\bar{\lambda}_3)P(\bar{\lambda}_4)}\Big[\frac{e_1(\bar{\lambda}_1)^2}
				{(\bar{\lambda}_1-\bar{\lambda}_3)(\bar{\lambda}_1-\bar{\lambda}_4)}+
				\frac{e_2(\bar{\lambda}_1)^2}{(\bar{\lambda}_2-\bar{\lambda}_3)(\bar{\lambda}_2-
					\bar{\lambda}_4)}\Big].
			\end{split}
		\end{equation}
		From (\ref{con3-1}), for $i\in [\bar{\lambda}_3]$,
		\begin{equation}\label{pro2-6} R_{1i1i}=\frac{e_1(h_{ii,1})}{\bar{\lambda}_3-\bar{\lambda}_1}+\frac{2h_{1i,2}^2}{(\bar{\lambda}_2-\bar{\lambda}_1)(\bar{\lambda}_2-\bar{\lambda}_3)}
			-\frac{h_{ii,2}h_{11,2}}{(\bar{\lambda}_2-\bar{\lambda}_1)(\bar{\lambda}_2-\bar{\lambda}_3)}
			+\frac{h_{11,1}h_{ii,1}-2h_{ii,1}^2}{(\bar{\lambda}_3-\bar{\lambda}_1)^2}.
		\end{equation}

		Since $dim D_{\bar{\lambda}_3}\geq 2$, we can assume that $h_{12,i_0}=0$ for some $i_0\in [\bar{\lambda}_3]$, otherwise, we can reselect $\{e_3,\cdots,e_{2+m_2}\}$ such that $h_{12,4}=0$. In fact, defining
$$X=\sum_{i\in [\bar{\lambda}_3]}h_{12,i}e_i.$$
If $X\neq 0$, taking the orthonormal frame $\bar{e}_3=\frac{X}{|X|}, \bar{e}_4, \cdots, \bar{e}_3\in D_{\bar{\lambda}_3}$, then $h_{12,4}=0.$
		\begin{equation*}
			R_{1i_01i_0}=\frac{e_1(h_{ii,1})}{\bar{\lambda}_3-\bar{\lambda}_1}
			-\frac{h_{ii,2}h_{11,2}}{(\bar{\lambda}_2-\bar{\lambda}_1)(\bar{\lambda}_2-\bar{\lambda}_3)}
			+\frac{h_{11,1}h_{ii,1}-2h_{ii,1}^2}{(\bar{\lambda}_3-\bar{\lambda}_1)^2}.
		\end{equation*}
		And $R_{1i1i}=R_{1i_01i_0}$, thus we get
		\begin{equation}\label{pro2-7}
			h_{12,i}=0, ~~\forall i\in [\bar{\lambda_3}].
		\end{equation}
		Similarly,
		\begin{equation}
			h_{12,j}=0, ~~\forall j\in [\bar{\lambda_4}].
		\end{equation}

Thus we obtain,
\begin{equation*}
			R_{1i1i}=\frac{e_1(h_{ii,1})}{\bar{\lambda}_3-\bar{\lambda}_1}
			-\frac{h_{ii,2}h_{11,2}}{(\bar{\lambda}_2-\bar{\lambda}_1)(\bar{\lambda}_2-\bar{\lambda}_3)}
			+\frac{h_{11,1}h_{ii,1}-2h_{ii,1}^2}{(\bar{\lambda}_3-\bar{\lambda}_1)^2},
		\end{equation*}
and
\begin{equation*}
			R_{1j1j}=\frac{e_1(h_{jj,1})}{\bar{\lambda}_4-\bar{\lambda}_1}
			-\frac{h_{jj,2}h_{11,2}}{(\bar{\lambda}_2-\bar{\lambda}_1)(\bar{\lambda}_2-\bar{\lambda}_4)}
			+\frac{h_{11,1}h_{jj,1}-2h_{jj,1}^2}{(\bar{\lambda}_4-\bar{\lambda}_1)^2}.
		\end{equation*}

		Combining (\ref{con41}),  we have
		\begin{equation}\label{pro2-5}
			\begin{split}
				e_1e_1(\bar{\lambda}_3)&=(\bar{\lambda}_3-\bar{\lambda}_1)R_{1i1i}
				 +\frac{P(\bar{\lambda}_1)(\bar{\lambda}_3-\bar{\lambda}_1)}{P(\bar{\lambda}_3)(\bar{\lambda}_2-\bar{\lambda}_1)(\bar{\lambda}_2-\bar{\lambda}_3)}e_2(\bar{\lambda}_1)^2\\
				&+\frac{P(\bar{\lambda}_1)P(\bar{\lambda}_3)-2P(\bar{\lambda}_1)^2}{P(\bar{\lambda}_3)^2(\bar{\lambda}_1-\bar{\lambda}_3)}
				e_1(\bar{\lambda}_1)^2,
			\end{split}
		\end{equation}
		and
		\begin{equation}\label{pro2-6}
			\begin{split}
				e_1e_1(\bar{\lambda}_4)&=(\bar{\lambda}_4-\bar{\lambda}_1)R_{1j1j}
				 +\frac{P(\bar{\lambda}_1)(\bar{\lambda}_4-\bar{\lambda}_1)}{P(\bar{\lambda}_4)(\bar{\lambda}_2-\bar{\lambda}_1)(\bar{\lambda}_2-\bar{\lambda}_4)}e_2(\bar{\lambda}_1)^2\\
				&+\frac{P(\bar{\lambda}_1)P(\bar{\lambda}_4)-2P(\bar{\lambda}_1)^2}{P(\bar{\lambda}_4)^2(\bar{\lambda}_1-\bar{\lambda}_4)}
				e_1(\bar{\lambda}_1)^2.
			\end{split}
		\end{equation}
		
		Similarly, we have
		\begin{equation}\label{pro2-7}
			\begin{split}
				e_2e_2(\bar{\lambda}_3)&=(\bar{\lambda}_3-\bar{\lambda}_2)R_{2i2i}
				+\frac{P(\bar{\lambda}_1)^2(\bar{\lambda}_3-\bar{\lambda}_2)}{P(\bar{\lambda}_2)P(\bar{\lambda}_3)(\bar{\lambda}_1
					-\bar{\lambda}_2)(\bar{\lambda}_1-\bar{\lambda}_3)}e_1(\bar{\lambda}_1)^2\\
				 &+\frac{P(\bar{\lambda}_1)^2P(\bar{\lambda}_3)-2P(\bar{\lambda}_1)^2P(\bar{\lambda}_2)}{P(\bar{\lambda}_2)P(\bar{\lambda}_3)^2(\bar{\lambda}_2-\bar{\lambda}_3)}
				e_2(\bar{\lambda}_1)^2,
			\end{split}
		\end{equation}
		and
		\begin{equation}\label{pro2-8}
			\begin{split}
				e_2e_2(\bar{\lambda}_4)&=(\bar{\lambda}_4-\bar{\lambda}_2)R_{2j2j}
				+\frac{P(\bar{\lambda}_1)^2(\bar{\lambda}_4-\bar{\lambda}_2)}{P(\bar{\lambda}_2)P(\bar{\lambda}_4)(\bar{\lambda}_1
					-\bar{\lambda}_2)(\bar{\lambda}_1-\bar{\lambda}_4)}e_1(\bar{\lambda}_1)^2\\
				 &+\frac{P(\bar{\lambda}_1)^2P(\bar{\lambda}_4)-2P(\bar{\lambda}_1)^2P(\bar{\lambda}_2)}{P(\bar{\lambda}_2)P(\bar{\lambda}_4)^2(\bar{\lambda}_2-\bar{\lambda}_4)}
				e_2(\bar{\lambda}_1)^2.
			\end{split}
		\end{equation}
		Let
		$$U(\bar{\lambda}_3,\bar{\lambda}_4)=\sum_{m\neq 3}\frac{P(\bar{\lambda}_3)-P(\bar{\lambda}_m)}
		{P(\bar{\lambda}_3)P(\bar{\lambda}_m)(\bar{\lambda}_m-\bar{\lambda}_3)}-\sum_{m\neq 4}\frac{P(\bar{\lambda}_4)-P(\bar{\lambda}_m)}
		{P(\bar{\lambda}_4)P(\bar{\lambda}_m)(\bar{\lambda}_m-\bar{\lambda}_4)}.$$
		By (\ref{cond5}), we get
		\begin{equation}\label{pro2-9}
			P(\bar{\lambda}_4)e_ke_k(\bar{\lambda}_4)-P(\bar{\lambda}_3)e_ke_k(\bar{\lambda}_3)
			=P(\bar{\lambda}_1)^2e_k(\bar{\lambda}_1)^2U(\bar{\lambda}_3,\bar{\lambda}_4).
		\end{equation}
Let $k=1,2$ in (\ref{pro2-9}), we get
$$P(\bar{\lambda}_4)e_1e_1(\bar{\lambda}_4)-P(\bar{\lambda}_3)e_1e_1(\bar{\lambda}_3)
			=P(\bar{\lambda}_1)^2e_1(\bar{\lambda}_1)^2U(\bar{\lambda}_3,\bar{\lambda}_4),$$
and
$$P(\bar{\lambda}_4)e_2e_2(\bar{\lambda}_4)-P(\bar{\lambda}_3)e_2e_2(\bar{\lambda}_3)
			=P(\bar{\lambda}_1)^2e_2(\bar{\lambda}_1)^2U(\bar{\lambda}_3,\bar{\lambda}_4).$$
		
		Combining (\ref{pro2-5}), (\ref{pro2-6}), (\ref{pro2-7}), (\ref{pro2-8}) and (\ref{pro2-9}), we get
		\begin{equation}\label{pro2-10}
			\begin{split}
				&P(\bar{\lambda}_4)(\bar{\lambda}_4-\bar{\lambda}_1)R_{1j1j}-P(\bar{\lambda}_3)(\bar{\lambda}_3-\bar{\lambda}_1)R_{1i1i}\\
				&=\Big[U(\bar{\lambda}_3,\bar{\lambda}_4)-\frac{P(\bar{\lambda}_4)-2P(\bar{\lambda}_1)}
				{P(\bar{\lambda}_4)P(\bar{\lambda}_1)(\bar{\lambda}_1-\bar{\lambda}_4)}+\frac{P(\bar{\lambda}_3)-2P(\bar{\lambda}_1)}
				{P(\bar{\lambda}_3)P(\bar{\lambda}_1)(\bar{\lambda}_1-\bar{\lambda}_3)}\Big]P(\bar{\lambda}_1)^2e_1(\bar{\lambda}_1)^2\\
				&+\frac{\bar{\lambda}_3-\bar{\lambda}_4}{P(\bar{\lambda}_1)(\bar{\lambda}_2-\bar{\lambda}_3)(\bar{\lambda}_2-\bar{\lambda}_4)}
				P(\bar{\lambda}_1)^2e_2(\bar{\lambda}_1)^2,
			\end{split}
		\end{equation}
		and
		\begin{equation}\label{pro2-11}
			\begin{split}
				&P(\bar{\lambda}_4)(\bar{\lambda}_4-\bar{\lambda}_2)R_{2j2j}-P(\bar{\lambda}_3)(\bar{\lambda}_3-\bar{\lambda}_2)R_{2i2i}\\
				&=\Big[U(\bar{\lambda}_3,\bar{\lambda}_4)-\frac{P(\bar{\lambda}_4)-2P(\bar{\lambda}_2)}
				{P(\bar{\lambda}_4)P(\bar{\lambda}_1)(\bar{\lambda}_2-\bar{\lambda}_4)}+\frac{P(\bar{\lambda}_3)-2P(\bar{\lambda}_2)}
				{P(\bar{\lambda}_3)P(\bar{\lambda}_1)(\bar{\lambda}_2-\bar{\lambda}_3)}\Big]P(\bar{\lambda}_1)^2e_2(\bar{\lambda}_1)^2\\
				&+\frac{\bar{\lambda}_3-\bar{\lambda}_4}{P(\bar{\lambda}_2)(\bar{\lambda}_1-\bar{\lambda}_3)(\bar{\lambda}_1-\bar{\lambda}_4)}
				P(\bar{\lambda}_1)^2e_1(\bar{\lambda}_1)^2.
			\end{split}
		\end{equation}

Let
$$Q_1=U(\bar{\lambda}_3,\bar{\lambda}_4)-\frac{P(\bar{\lambda}_4)-2P(\bar{\lambda}_1)}
				{P(\bar{\lambda}_4)P(\bar{\lambda}_1)(\bar{\lambda}_1-\bar{\lambda}_4)}+\frac{P(\bar{\lambda}_3)-2P(\bar{\lambda}_1)}
				{P(\bar{\lambda}_3)P(\bar{\lambda}_1)(\bar{\lambda}_1-\bar{\lambda}_3)},$$
$$Q_2=U(\bar{\lambda}_3,\bar{\lambda}_4)-\frac{P(\bar{\lambda}_4)-2P(\bar{\lambda}_2)}
				{P(\bar{\lambda}_4)P(\bar{\lambda}_1)(\bar{\lambda}_2-\bar{\lambda}_4)}+\frac{P(\bar{\lambda}_3)-2P(\bar{\lambda}_2)}
				{P(\bar{\lambda}_3)P(\bar{\lambda}_1)(\bar{\lambda}_2-\bar{\lambda}_3)}.$$
Then the equation (\ref{pro2-10}) and (\ref{pro2-11}) can be rewritten as the following equation,
\begin{equation}\label{pro2-12}
\begin{split}
&\frac{Q_1P(\bar{\lambda}_1)^2}{\bar{\lambda_3}-\bar{\lambda}_4}e_1(\bar{\lambda}_1)^2
				+\frac{1}{(\bar{\lambda}_2-\bar{\lambda}_3)(\bar{\lambda}_2-\bar{\lambda}_4)}e_2(\bar{\lambda}_1)^2\\
&=\frac{1}{P(\bar{\lambda}_1)(\bar{\lambda_3}-\bar{\lambda}_4)}\Big(P(\bar{\lambda}_4)(\bar{\lambda}_4-\bar{\lambda}_1)R_{1j1j}
-P(\bar{\lambda}_3)(\bar{\lambda}_3-\bar{\lambda}_1)R_{1i1i}\Big),
\end{split}
\end{equation}
and
\begin{equation}\label{pro2-13}
\begin{split}
&\frac{P(\bar{\lambda}_2)Q_2}{\bar{\lambda}_3-\bar{\lambda}_4}e_2(\bar{\lambda}_1)^2
				+\frac{1}{(\bar{\lambda}_1-\bar{\lambda}_3)(\bar{\lambda}_1-\bar{\lambda}_4)}e_1(\bar{\lambda}_1)^2\\
&=\frac{P(\bar{\lambda}_2)}{P(\bar{\lambda}_1)^2(\bar{\lambda}_3-\bar{\lambda}_4)}\Big(P(\bar{\lambda}_4)(\bar{\lambda}_4-\bar{\lambda}_2)R_{2j2j}
                 -P(\bar{\lambda}_3)(\bar{\lambda}_3-\bar{\lambda}_2)R_{2i2i}\Big).
			\end{split}
		\end{equation}
From (\ref{pro2-12}), we get
\begin{equation}\label{pro2-14}
\frac{e_1(\bar{\lambda}_1)^2}
				{(\bar{\lambda}_1-\bar{\lambda}_3)(\bar{\lambda}_1-\bar{\lambda}_4)}+
				\frac{e_2(\bar{\lambda}_1)^2}{(\bar{\lambda}_2-\bar{\lambda}_3)(\bar{\lambda}_2-\bar{\lambda}_4)}
=\frac{-P(\bar{\lambda}_3)P(\bar{\lambda}_4)}{P(\bar{\lambda}_1)^2}R_{ijij}.
		\end{equation}

Let
$$\alpha=\frac{1}{P(\bar{\lambda}_1)(\bar{\lambda_3}-\bar{\lambda}_4)}\Big(P(\bar{\lambda}_4)(\bar{\lambda}_4-\bar{\lambda}_1)R_{1j1j}
-P(\bar{\lambda}_3)(\bar{\lambda}_3-\bar{\lambda}_1)R_{1i1i}\Big)+\frac{P(\bar{\lambda}_3)P(\bar{\lambda}_4)}{P(\bar{\lambda}_1)^2}R_{ijij}.$$
$$\beta=\frac{P(\bar{\lambda}_2)}{P(\bar{\lambda}_1)^2(\bar{\lambda}_3-\bar{\lambda}_4)}\Big(P(\bar{\lambda}_4)(\bar{\lambda}_4-\bar{\lambda}_2)R_{2j2j}
                 -P(\bar{\lambda}_3)(\bar{\lambda}_3
                 -\bar{\lambda}_2)R_{2i2i}\Big)+\frac{P(\bar{\lambda}_3)P(\bar{\lambda}_4)}{P(\bar{\lambda}_1)^2}R_{ijij}.$$

Combining (\ref{pro2-12}) and (\ref{pro2-14}), we obtain,
\begin{equation}\label{pro2-15}
\Big(\frac{Q_1P(\bar{\lambda}_1)^2}{\bar{\lambda_3}-\bar{\lambda}_4}-\frac{1}
				{(\bar{\lambda}_1-\bar{\lambda}_3)(\bar{\lambda}_1-\bar{\lambda}_4)}\Big)e_1(\bar{\lambda}_1)^2=\alpha.
\end{equation}

Combining (\ref{pro2-13}) and (\ref{pro2-14}), we obtain,
\begin{equation}\label{pro2-16}
\Big(\frac{P(\bar{\lambda}_2)Q_2}{\bar{\lambda}_3-\bar{\lambda}_4}
				-\frac{1}{(\bar{\lambda}_2-\bar{\lambda}_3)(\bar{\lambda}_2-\bar{\lambda}_4)}\Big)e_2(\bar{\lambda}_1)^2=\beta.
		\end{equation}
By Gauss equations, $P(\bar{\lambda}_i)$, $\alpha$ and $\beta$ are  rational polynomials about variables $\bar{\lambda}_1,\bar{\lambda}_2, \bar{\lambda}_3$ and $\bar{\lambda}_4$.

If $$\frac{Q_1P(\bar{\lambda}_1)^2}{\bar{\lambda_3}-\bar{\lambda}_4}-\frac{1}
				{(\bar{\lambda}_1-\bar{\lambda}_3)(\bar{\lambda}_1-\bar{\lambda}_4)}=0,$$ or $$\frac{P(\bar{\lambda}_2)Q_2}{\bar{\lambda}_3-\bar{\lambda}_4}
				-\frac{1}{(\bar{\lambda}_2-\bar{\lambda}_3)(\bar{\lambda}_2-\bar{\lambda}_4)}=0.$$

 Combining the three  algebraic equations about four variables $\bar{\lambda}_1, \bar{\lambda}_2, \bar{\lambda}_3$ and $\bar{\lambda}_4$,
 $$\bar{\lambda}_1^k+\bar{\lambda}_2^k+m_3\bar{\lambda}_3^k+m_4\bar{\lambda}_4^k=c_k, ~~ k=1,2,3,$$
 where $c_1,c_2,c_3$ are constant,
 then we obtain four rational algebraic equations about four variables $\bar{\lambda}_1, \bar{\lambda}_2, \bar{\lambda}_3$ and $\bar{\lambda}_4$, which are independent of each other. Hence $\bar{\lambda}_1,\bar{\lambda}_2,\bar{\lambda}_3,\bar{\lambda}_4$ are constant and $$e_1(\bar{\lambda}_1)=\cdots=e_1(\bar{\lambda}_g)= 0,~~e_2(\bar{\lambda}_1)=\cdots=e_2(\bar{\lambda}_g)= 0.$$
		This is a  contradiction.
		
If $$\frac{Q_1P(\bar{\lambda}_1)^2}{\bar{\lambda_3}-\bar{\lambda}_4}-\frac{1}
				{(\bar{\lambda}_1-\bar{\lambda}_3)(\bar{\lambda}_1-\bar{\lambda}_4)}\neq0,$$ and $$\frac{P(\bar{\lambda}_2)Q_2}{\bar{\lambda}_3-\bar{\lambda}_4}
				-\frac{1}{(\bar{\lambda}_2-\bar{\lambda}_3)(\bar{\lambda}_2-\bar{\lambda}_4)}\neq0.$$		
Let
$$\mu=\frac{Q_1P(\bar{\lambda}_1)^2}{\bar{\lambda_3}-\bar{\lambda}_4}-\frac{1}
				{(\bar{\lambda}_1-\bar{\lambda}_3)(\bar{\lambda}_1-\bar{\lambda}_4)},$$
$$\nu=\frac{P(\bar{\lambda}_2)Q_2}{\bar{\lambda}_3-\bar{\lambda}_4}
				-\frac{1}{(\bar{\lambda}_2-\bar{\lambda}_3)(\bar{\lambda}_2-\bar{\lambda}_4)}.$$

Combining (\ref{pro2-14}), (\ref{pro2-15}) and (\ref{pro2-16}), we obtain,
\begin{equation}\label{pro2-17}
\frac{\nu\alpha}{(\bar{\lambda}_1
-\bar{\lambda}_3)(\bar{\lambda}_1-\bar{\lambda}_4)}
+\frac{\mu\beta}{(\bar{\lambda}_2
-\bar{\lambda}_3)(\bar{\lambda}_2-\bar{\lambda}_4)}
=\frac{-P(\bar{\lambda}_3)P(\bar{\lambda}_4)}{P(\bar{\lambda}_1)^2}R_{ijij}\mu\nu.
		\end{equation}
Since $\mu,\nu$, $\alpha,\beta$ and $P(\bar{\lambda}_i)$  are  rational polynomials about variables $\bar{\lambda}_1,\bar{\lambda}_2, \bar{\lambda}_3$ and $\bar{\lambda}_4$, the equation (\ref{pro2-17}) is rational polynomials about variables $\bar{\lambda}_1,\bar{\lambda}_2, \bar{\lambda}_3$ and $\bar{\lambda}_4$. Combining the three  algebraic equations about four variables $\bar{\lambda}_1, \bar{\lambda}_2, \bar{\lambda}_3$ and $\bar{\lambda}_4$,
 $$\bar{\lambda}_1^k+\bar{\lambda}_2^k+m_3\bar{\lambda}_3^k+m_4\bar{\lambda}_4^k=c_k, ~~ k=1,2,3,$$
 where $c_1,c_2,c_3$ are constant,
 then we obtain four rational algebraic equations about four variables $\bar{\lambda}_1, \bar{\lambda}_2, \bar{\lambda}_3$ and $\bar{\lambda}_4$, which are independent of each other. Hence $\bar{\lambda}_1,\bar{\lambda}_2,\bar{\lambda}_3,\bar{\lambda}_4$ are constant and $$e_1(\bar{\lambda}_1)=\cdots=e_1(\bar{\lambda}_g)= 0,~~e_2(\bar{\lambda}_1)=\cdots=e_2(\bar{\lambda}_g)= 0.$$
		This is a  contradiction.

If we assume that $$e_1(\bar{\lambda}_1),\cdots,e_1(\bar{\lambda}_g)\neq 0,~~e_2(\bar{\lambda}_1)=\cdots=e_2(\bar{\lambda}_g)= 0,$$ then
		there exist at least $g-1$ principal curvatures  that are Dupin principal curvatures. By theorem \ref{th3}, this is a contradiction, thus  we prove  {\bf Case 3}.
		
		Next we prove {\bf Case 4},
		$$(h_{ij})=diag(\lambda_1,\lambda_2,\cdots,\lambda_n)=diag(\bar{\lambda}_1, \bar{\lambda}_2,
		\bar{\lambda}_3,\underbrace{\bar{\lambda}_4,\cdots,\bar{\lambda}_4}_{m_4}).$$
		Next we make the following convention
		on the range of indices, $$4\leq \alpha,\beta,\gamma\leq n.$$
		Thus we have
		$$e_{\alpha}(\bar{\lambda}_1)=e_{\alpha}(\bar{\lambda}_2)=e_{\alpha}(\bar{\lambda}_3)=e_{\alpha}(\bar{\lambda}_4)=0.$$
		We fix the index $\alpha,\beta\in [\bar{\lambda}_4]$ and $\alpha\neq\beta$
		by the second covariant derivative of the second fundamental form and (\ref{con2-1}), we obtain the following equations,
		\begin{equation*}
			\begin{split}
				&h_{11,\alpha\beta}=\frac{2h_{12,\alpha}h_{12,\beta}}{\bar{\lambda}_2-\bar{\lambda}_1}
				+\frac{2h_{13,\alpha}h_{13,\beta}}{\bar{\lambda}_3-\bar{\lambda}_1}\\
				 &h_{\alpha\beta,11}=\frac{2h_{12,\alpha}h_{12,\beta}}{\bar{\lambda}_2-\bar{\lambda_4}}+\frac{2h_{13,\alpha}h_{13,\beta}}{\bar{\lambda}_3-\bar{\lambda_4}}
			\end{split}
		\end{equation*}
		By Ricci identity, we get
		\begin{equation}\label{pro3-1}
			\frac{h_{12,\alpha}h_{12,\beta}}{(\bar{\lambda}_2-\bar{\lambda}_1)(\bar{\lambda}_2-\bar{\lambda}_4)}
			+\frac{h_{13,\alpha}h_{13,\beta}}{(\bar{\lambda}_3-\bar{\lambda}_1)(\bar{\lambda}_3-\bar{\lambda}_4)}=0.
		\end{equation}
		similarly we have
		\begin{equation}\label{pro3-2}
			\begin{split}
				&\frac{h_{12,\alpha}h_{12,\beta}}{(\bar{\lambda}_1-\bar{\lambda}_2)(\bar{\lambda}_1-\bar{\lambda}_4)}
				+\frac{h_{23,\alpha}h_{23,\beta}}{(\bar{\lambda}_3-\bar{\lambda}_2)(\bar{\lambda}_3-\bar{\lambda}_4)}=0,\\
				&\frac{h_{13,\alpha}h_{13,\beta}}{(\bar{\lambda}_1-\bar{\lambda}_3)(\bar{\lambda}_1-\bar{\lambda}_4)}
				+\frac{h_{23,\alpha}h_{23,\beta}}{(\bar{\lambda}_2-\bar{\lambda}_3)(\bar{\lambda}_2-\bar{\lambda}_4)}=0.
			\end{split}
		\end{equation}
		{\bf Claim}: $h_{12,\alpha}=0,~~h_{13,\alpha}=0,~~h_{23,\alpha}=0, ~~4\leq \alpha\leq n$.
		
		We prove Claim  by contradiction. We can assume $h_{12,\alpha}\neq0$ for some $\alpha$. Thus the vector $E=\sum_{\gamma}h_{12,\gamma}e_{\gamma}\neq 0$, and we can reselect orthonormal basis $\{e_4,\cdots,e_n\}$ in $D_{\bar{\lambda}_4}$ such that $e_4=\frac{E}{|E|}$. Thus we have
		$$h_{12,4}\neq 0,~h_{12,5}=\cdots=h_{12,n}=0.$$
		Combining (\ref{pro3-2}), we obtain
		\begin{equation}\label{pro3-3}
			h_{12,\alpha}h_{12,\beta}=h_{13,\alpha}h_{13,\beta}=h_{23,\alpha}h_{23,\beta}=0,~~\alpha\neq\beta.
		\end{equation}
		So for $\alpha\ne \beta$, there are eight cases in total:
		\begin{enumerate}
			\item $h_{12\alpha}=0,\,h_{13\alpha}=0,\,h_{23\alpha}=0$,
			\item $h_{12\alpha}=0,\,h_{13\alpha}=0,\,h_{23\beta}=0$,
			\item $h_{12\alpha}=0,\,h_{13\beta}=0,\,h_{23\alpha}=0$,
			\item $h_{12\alpha}=0,\,h_{13\beta}=0,\,h_{23\beta}=0$,
			\item $h_{12\beta}=0,\,h_{13\alpha}=0,\,h_{23\alpha}=0$,
			\item $h_{12\beta}=0,\,h_{13\alpha}=0,\,h_{23\beta}=0$,
			\item $h_{12\beta}=0,\,h_{13\beta}=0,\,h_{23\alpha}=0$,
			\item $h_{12\beta}=0,\,h_{13\beta}=0,\,h_{23\beta}=0$.
		\end{enumerate}
		
		Thus there exists at least one index $\beta$ that satifies either $h_{12,\beta}=h_{13,\beta}=0$,
		$h_{12,\beta}=h_{23,\beta}=0$, or $h_{13,\beta}=h_{23,\beta}=0$. Without loss of generality, we assume
		$$h_{12,\beta}=h_{13,\beta}=0.$$
		In (\ref{con3-1}), Let $i=1$ and $j=\beta$, we have
		\begin{equation}\label{pro3-40}
			\begin{split}
				R_{1\beta1\beta}&=\frac{e_1(h_{\beta\beta,1})}{\bar{\lambda}_4-\bar{\lambda}_1}
				-\frac{h_{11,2}h_{\beta\beta,2}}{(\bar{\lambda}_2-\bar{\lambda}_1)(\bar{\lambda}_2-\bar{\lambda}_4)}\\
				&-\frac{h_{11,3}h_{\beta\beta,3}}{(\bar{\lambda}_3-\bar{\lambda}_1)(\bar{\lambda}_3-\bar{\lambda}_4)}
				+\frac{h_{11,1}h_{\beta\beta,1}-2h_{\beta\beta,1}^2}{(\bar{\lambda}_1-\bar{\lambda}_4)^2}.
			\end{split}
		\end{equation}
		In (\ref{con3-1}), Let $i=1$ and $j=\gamma\neq\beta$, we have
		\begin{equation}\label{pro3-4}
			\begin{split}
				R_{1\gamma1\gamma}&=\frac{e_1(h_{\gamma\gamma,1})}{\bar{\lambda}_4-\bar{\lambda}_1}+
				\frac{2h_{12,\gamma}^2}{(\bar{\lambda}_2-\bar{\lambda}_1)(\bar{\lambda}_2-\bar{\lambda}_4)}+
				\frac{2h_{13,\gamma}^2}{(\bar{\lambda}_3-\bar{\lambda}_1)(\bar{\lambda}_3-\bar{\lambda}_4)}\\
				&-\frac{h_{11,2}h_{\gamma\gamma,2}}{(\bar{\lambda}_2-\bar{\lambda}_1)(\bar{\lambda}_2-\bar{\lambda}_4)}
				-\frac{h_{11,3}h_{\gamma\gamma,3}}{(\bar{\lambda}_3-\bar{\lambda}_1)(\bar{\lambda}_3-\bar{\lambda}_4)}
				+\frac{h_{11,1}h_{\gamma\gamma,1}-2h_{\gamma\gamma,1}^2}{(\bar{\lambda}_1-\bar{\lambda}_4)^2}.
			\end{split}
		\end{equation}
		Since $h_{\beta\beta,1}=h_{\gamma\gamma,1}=e_1(\bar{\lambda_4})$, by (\ref{pro3-4}) and (\ref{pro3-40})
		we have
		$$\frac{h_{12,\gamma}^2}{(\bar{\lambda}_2-\bar{\lambda}_1)(\bar{\lambda}_2-\bar{\lambda}_4)}+
		\frac{h_{13,\gamma}^2}{(\bar{\lambda}_3-\bar{\lambda}_1)(\bar{\lambda}_3-\bar{\lambda}_4)}=0.$$

		So  $h_{12,\gamma}=0$ and $h_{13,\gamma}=0$,	which is a contradiction. Thus we finish the proof of Claim.
		
		By the Claim, let $i,j=1,2,3$ in (\ref{con3-1}), we can obtain the following equations,
		\begin{equation}\label{pro3-71}
			\begin{split}
				&e_1e_1(\bar{\lambda}_2)-e_2e_2(\bar{\lambda}_1)=
				\frac{2(\bar{\lambda}_1-\bar{\lambda}_2)h_{12,3}^2}{(\bar{\lambda}_3-\bar{\lambda}_1)(\bar{\lambda}_3-\bar{\lambda}_2)}
				 +\frac{(\bar{\lambda}_2-\bar{\lambda}_1)e_3(\bar{\lambda}_1)e_3(\bar{\lambda}_2)}{(\bar{\lambda}_3-\bar{\lambda}_1)(\bar{\lambda}_3-\bar{\lambda}_2)}\\
				&+(\bar{\lambda}_2-\bar{\lambda}_1)R_{1212}+\frac{e_1(\bar{\lambda}_1)e_1(\bar{\lambda}_2)+e_2(\bar{\lambda}_1)e_2(\bar{\lambda}_2)
					-2(e_1(\bar{\lambda}_2)^2+e_2(\bar{\lambda}_1)^2)}{\bar{\lambda}_1-\bar{\lambda}_2};
			\end{split}
		\end{equation}
		\begin{equation}\label{pro3-72}
			\begin{split}
				&e_1e_1(\bar{\lambda}_3)-e_3e_3(\bar{\lambda}_1)=
				\frac{2(\bar{\lambda}_1-\bar{\lambda}_3)h_{12,3}^2}{(\bar{\lambda}_2-\bar{\lambda}_1)(\bar{\lambda}_2-\bar{\lambda}_3)}
				 +\frac{(\bar{\lambda}_3-\bar{\lambda}_1)e_2(\bar{\lambda}_1)e_2(\bar{\lambda}_3)}{(\bar{\lambda}_2-\bar{\lambda}_1)(\bar{\lambda}_2-\bar{\lambda}_3)}\\
				&+(\bar{\lambda}_3-\bar{\lambda}_1)R_{1313}+\frac{e_1(\bar{\lambda}_1)e_1(\bar{\lambda}_3)+e_3(\bar{\lambda}_1)e_3(\bar{\lambda}_3)
					-2(e_1(\bar{\lambda}_3)^2+e_3(\bar{\lambda}_1)^2)}{\bar{\lambda}_1-\bar{\lambda}_3};
			\end{split}
		\end{equation}
		\begin{equation}\label{pro3-73}
			\begin{split}
				&e_2e_2(\bar{\lambda}_3)-e_3e_3(\bar{\lambda}_2)=
				+\frac{2(\bar{\lambda}_2-\bar{\lambda}_3)h_{12,3}^2}{(\bar{\lambda}_1-\bar{\lambda}_2)(\bar{\lambda}_1-\bar{\lambda}_3)}
				 +\frac{(\bar{\lambda}_3-\bar{\lambda}_2)e_1(\bar{\lambda}_2)e_1(\bar{\lambda}_3)}{(\bar{\lambda}_1-\bar{\lambda}_2)(\bar{\lambda}_1-\bar{\lambda}_3)}\\
				&+(\bar{\lambda}_3-\bar{\lambda}_2)R_{2323}+\frac{e_2(\bar{\lambda}_2)e_2(\bar{\lambda}_3)+e_3(\bar{\lambda}_2)e_3(\bar{\lambda}_3)
					-2(e_2(\bar{\lambda}_3)^2+e_3(\bar{\lambda}_2)^2)}{\bar{\lambda}_2-\bar{\lambda}_3}.
			\end{split}
		\end{equation}
		
		Let $i=1,2,3$ and $j=\alpha$ in (\ref{con3-1}), we can obtain the following equations,
		\begin{equation}\label{pro3-6}
			\begin{split}
				e_1e_1(\bar{\lambda}_4)&=(\bar{\lambda}_4-\bar{\lambda}_1)R_{1\alpha1\alpha}
				 +\frac{(\bar{\lambda}_4-\bar{\lambda}_1)e_2(\bar{\lambda}_1)e_2(\bar{\lambda}_4)}{(\bar{\lambda}_2-\bar{\lambda}_1)(\bar{\lambda}_2-\bar{\lambda}_4)}\\
				 &+\frac{(\bar{\lambda}_4-\bar{\lambda}_1)e_3(\bar{\lambda}_1)e_3(\bar{\lambda}_4)}{(\bar{\lambda}_3-\bar{\lambda}_1)(\bar{\lambda}_3-\bar{\lambda}_4)}
				+\frac{e_1(\bar{\lambda}_1)e_1(\bar{\lambda}_4)-2e_1(\bar{\lambda}_4)^2}{\bar{\lambda}_1-\bar{\lambda}_4};\\
				e_2e_2(\bar{\lambda}_4)&=(\bar{\lambda}_4-\bar{\lambda}_2)R_{2\alpha2\alpha}
				 +\frac{(\bar{\lambda}_4-\bar{\lambda}_2)e_1(\bar{\lambda}_2)e_1(\bar{\lambda}_4)}{(\bar{\lambda}_1-\bar{\lambda}_2)(\bar{\lambda}_1-\bar{\lambda}_4)}\\
				 &+\frac{(\bar{\lambda}_4-\bar{\lambda}_2)e_3(\bar{\lambda}_2)e_3(\bar{\lambda}_4)}{(\bar{\lambda}_3-\bar{\lambda}_2)(\bar{\lambda}_3-\bar{\lambda}_4)}
				+\frac{e_2(\bar{\lambda}_2)e_2(\bar{\lambda}_4)-2e_2(\bar{\lambda}_4)^2}{\bar{\lambda}_2-\bar{\lambda}_4};\\
				e_3e_3(\bar{\lambda}_4)&=(\bar{\lambda}_4-\bar{\lambda}_3)R_{3\alpha3\alpha}
				 +\frac{(\bar{\lambda}_4-\bar{\lambda}_3)e_1(\bar{\lambda}_3)e_1(\bar{\lambda}_4)}{(\bar{\lambda}_1-\bar{\lambda}_3)(\bar{\lambda}_1-\bar{\lambda}_4)}\\
				 &+\frac{(\bar{\lambda}_4-\bar{\lambda}_3)e_2(\bar{\lambda}_3)e_2(\bar{\lambda}_4)}{(\bar{\lambda}_2-\bar{\lambda}_3)(\bar{\lambda}_2-\bar{\lambda}_4)}
				+\frac{e_3(\bar{\lambda}_3)e_3(\bar{\lambda}_4)-2e_3(\bar{\lambda}_4)^2}{\bar{\lambda}_3-\bar{\lambda}_4}.
			\end{split}
		\end{equation}
		
		By (\ref{cond5}), we have the following equations,
		\begin{equation}\label{pro3-81}
			\begin{split}
				&e_ie_i(\bar{\lambda}_1)=
				 -\frac{(\bar{\lambda}_4-\bar{\lambda}_2)(\bar{\lambda}_4-\bar{\lambda}_3)}{(\bar{\lambda}_2-\bar{\lambda}_1)(\bar{\lambda}_3-\bar{\lambda}_1)}
				m_4e_ie_i(\bar{\lambda}_4)\\
				&+e_i(\bar{\lambda}_1)\Big[\frac{e_i(\bar{\lambda}_2)-e_i(\bar{\lambda}_4)}{\bar{\lambda}_2-\bar{\lambda}_4}
				+\frac{e_i(\bar{\lambda}_3)-e_i(\bar{\lambda}_4)}{\bar{\lambda}_3-\bar{\lambda}_4}
				-\frac{e_i(\bar{\lambda}_2)-e_i(\bar{\lambda}_1)}{\bar{\lambda}_2-\bar{\lambda}_1}
				-\frac{e_i(\bar{\lambda}_3)-e_i(\bar{\lambda}_1)}{\bar{\lambda}_3-\bar{\lambda}_3}\Big];
			\end{split}
		\end{equation}
		\begin{equation}\label{pro3-82}
			\begin{split}
				&e_ie_i(\bar{\lambda}_2)=
				 -\frac{(\bar{\lambda}_4-\bar{\lambda}_1)(\bar{\lambda}_4-\bar{\lambda}_3)}{(\bar{\lambda}_1-\bar{\lambda}_2)(\bar{\lambda}_3-\bar{\lambda}_2)}
				m_4e_ie_i(\bar{\lambda}_4)\\
				&+e_i(\bar{\lambda}_2)\Big[\frac{e_i(\bar{\lambda}_1)-e_i(\bar{\lambda}_4)}{\bar{\lambda}_1-\bar{\lambda}_4}
				+\frac{e_i(\bar{\lambda}_3)-e_i(\bar{\lambda}_4)}{\bar{\lambda}_3-\bar{\lambda}_4}
				-\frac{e_i(\bar{\lambda}_1)-e_i(\bar{\lambda}_2)}{\bar{\lambda}_1-\bar{\lambda}_2}
				-\frac{e_i(\bar{\lambda}_3)-e_i(\bar{\lambda}_2)}{\bar{\lambda}_3-\bar{\lambda}_2}\Big];
			\end{split}
		\end{equation}
		\begin{equation}\label{pro3-83}
			\begin{split}
				&e_ie_i(\bar{\lambda}_3)=
				 -\frac{(\bar{\lambda}_4-\bar{\lambda}_1)(\bar{\lambda}_4-\bar{\lambda}_2)}{(\bar{\lambda}_1-\bar{\lambda}_3)(\bar{\lambda}_2-\bar{\lambda}_3)}
				m_4e_ie_i(\bar{\lambda}_4)\\
				&+e_i(\bar{\lambda}_3)\Big[\frac{e_i(\bar{\lambda}_1)-e_i(\bar{\lambda}_4)}{\bar{\lambda}_1-\bar{\lambda}_4}
				+\frac{e_i(\bar{\lambda}_2)-e_i(\bar{\lambda}_4)}{\bar{\lambda}_2-\bar{\lambda}_4}
				-\frac{e_i(\bar{\lambda}_1)-e_i(\bar{\lambda}_3)}{\bar{\lambda}_1-\bar{\lambda}_3}
				-\frac{e_i(\bar{\lambda}_2)-e_i(\bar{\lambda}_3)}{\bar{\lambda}_2-\bar{\lambda}_3}\Big].
			\end{split}
		\end{equation}
		By (\ref{cond4}), we have the following equations,
		\begin{equation}\label{pro3-5}
			\begin{split}
				&e_i(\bar{\lambda}_1)
				 =\frac{(\bar{\lambda}_4-\bar{\lambda}_2)(\bar{\lambda}_3-\bar{\lambda}_4)}{(\bar{\lambda}_2-\bar{\lambda}_1)(\bar{\lambda}_3-\bar{\lambda}_1)}
				m_4e_i(\bar{\lambda}_4),~~i=1,2,3;\\
				&e_i(\bar{\lambda}_2)
				 =\frac{(\bar{\lambda}_4-\bar{\lambda}_1)(\bar{\lambda}_3-\bar{\lambda}_4)}{(\bar{\lambda}_1-\bar{\lambda}_2)(\bar{\lambda}_3-\bar{\lambda}_2)}
				m_4e_i(\bar{\lambda}_4),~~i=1,2,3;\\
				&e_i(\bar{\lambda}_3)
				 =\frac{(\bar{\lambda}_4-\bar{\lambda}_1)(\bar{\lambda}_2-\bar{\lambda}_4)}{(\bar{\lambda}_1-\bar{\lambda}_3)(\bar{\lambda}_2-\bar{\lambda}_3)}
				m_4e_i(\bar{\lambda}_4),~~i=1,2,3.
			\end{split}
		\end{equation}	

Combining (\ref{pro3-6}), (\ref{pro3-81}), (\ref{pro3-82}), (\ref{pro3-83}) and (\ref{pro3-5}), we have the following equations.
	\begin{equation}\label{pro3-91}
		\begin{aligned}
			e_1[e_1(\bar{\lambda}_j)] &= e_1(\bar{\lambda}_4)^2
			\left[
			\sum_{l\neq 4} \frac{P(\bar{\lambda}_4) - P(\bar{\lambda}_l)}{P(\bar{\lambda}_l)(\bar{\lambda}_l - \bar{\lambda}_4)}
			- \sum_{l\neq j} \frac{P(\bar{\lambda}_4)\bigl(P(\bar{\lambda}_j) - P(\bar{\lambda}_l)\bigr)}{P(\bar{\lambda}_l)P(\bar{\lambda}_j)(\bar{\lambda}_l - \bar{\lambda}_j)}
			\right]
			\frac{P(\bar{\lambda}_4)}{P(\bar{\lambda}_j)} \\
			&\quad + \frac{P(\bar{\lambda}_4)}{P(\bar{\lambda}_j)}
			\left\{
			(\bar{\lambda}_4 - \bar{\lambda}_1)R_{1\alpha1\alpha}
			+ \frac{P(\bar{\lambda}_4)}{(\bar{\lambda}_2 - \bar{\lambda}_1)^2 (\bar{\lambda}_3 - \bar{\lambda}_1) (\bar{\lambda}_2 - \bar{\lambda}_4)} e_2(\bar{\lambda}_4)^2
			\right. \\
			&\left.
			\quad+ \frac{P(\bar{\lambda}_4)}{(\bar{\lambda}_2 - \bar{\lambda}_1) (\bar{\lambda}_3 - \bar{\lambda}_1)^2 (\bar{\lambda}_3 - \bar{\lambda}_4)} e_3(\bar{\lambda}_4)^2
			+ \frac{e_1(\bar{\lambda}_4)^2}{\bar{\lambda}_1 - \bar{\lambda}_4}\left(\frac{P(\bar{\lambda}_4)}{P(\bar{\lambda}_1)} - 2\right)
			\right\}, 		
		\end{aligned}
	\end{equation}
	\begin{equation}\label{pro3-92}
		\begin{aligned}
			e_2[e_2(\bar{\lambda}_j)] &= e_2(\bar{\lambda}_4)^2
			\left[
			\sum_{l\neq 4} \frac{P(\bar{\lambda}_4) - P(\bar{\lambda}_l)}{P(\bar{\lambda}_l)(\bar{\lambda}_l - \bar{\lambda}_4)}
			- \sum_{l\neq j} \frac{P(\bar{\lambda}_4)\bigl(P(\bar{\lambda}_j) - P(\bar{\lambda}_l)\bigr)}{P(\bar{\lambda}_l)P(\bar{\lambda}_j)(\bar{\lambda}_l - \bar{\lambda}_j)}
			\right]
			\frac{P(\bar{\lambda}_4)}{P(\bar{\lambda}_j)} \\
			&\quad + \frac{P(\bar{\lambda}_4)}{P(\bar{\lambda}_j)}
			\left\{
			(\bar{\lambda}_4 - \bar{\lambda}_2)R_{2\alpha2\alpha}
			+ \frac{P(\bar{\lambda}_4)}{(\bar{\lambda}_1 - \bar{\lambda}_2)^2 (\bar{\lambda}_3 - \bar{\lambda}_2) (\bar{\lambda}_1 - \bar{\lambda}_4)} e_1(\bar{\lambda}_4)^2
			\right. \\
			&
			\left.\quad
			+ \frac{P(\bar{\lambda}_4)}{(\bar{\lambda}_1 - \bar{\lambda}_2) (\bar{\lambda}_3 - \bar{\lambda}_2)^2 (\bar{\lambda}_3 - \bar{\lambda}_4)} e_3(\bar{\lambda}_4)^2
			+ \frac{e_2(\bar{\lambda}_4)^2}{\bar{\lambda}_2 - \bar{\lambda}_4}\left(\frac{P(\bar{\lambda}_4)}{P(\bar{\lambda}_2)} - 2\right)
			\right\},
		\end{aligned}
	\end{equation}
	\begin{equation}\label{pro3-93}
		\begin{aligned}
			e_3[e_3(\bar{\lambda}_j)] &= e_3(\bar{\lambda}_4)^2
			\left[
			\sum_{l\neq 4} \frac{P(\bar{\lambda}_4) - P(\bar{\lambda}_l)}{P(\bar{\lambda}_l)(\bar{\lambda}_l - \bar{\lambda}_4)}
			- \sum_{l\neq j} \frac{P(\bar{\lambda}_4)\bigl(P(\bar{\lambda}_j) - P(\bar{\lambda}_l)\bigr)}{P(\bar{\lambda}_l)P(\bar{\lambda}_j)(\bar{\lambda}_l - \bar{\lambda}_j)}
			\right]
			\frac{P(\bar{\lambda}_4)}{P(\bar{\lambda}_j)} \\
			&\quad + \frac{P(\bar{\lambda}_4)}{P(\bar{\lambda}_j)}
			\left\{
			(\bar{\lambda}_4 - \bar{\lambda}_3)R_{3\alpha3\alpha}
			+ \frac{P(\bar{\lambda}_4)}{(\bar{\lambda}_1 - \bar{\lambda}_3)^2 (\bar{\lambda}_2 - \bar{\lambda}_3) (\bar{\lambda}_1 - \bar{\lambda}_4)} e_1(\bar{\lambda}_4)^2
			\right. \\
			&
			\left.
			\quad+ \frac{P(\bar{\lambda}_4)}{(\bar{\lambda}_1 - \bar{\lambda}_3) (\bar{\lambda}_2 - \bar{\lambda}_3)^2 (\bar{\lambda}_2 - \bar{\lambda}_4)} e_2(\bar{\lambda}_4)^2
			+ \frac{e_3(\bar{\lambda}_4)^2}{\bar{\lambda}_3 - \bar{\lambda}_4}\left(\frac{P(\bar{\lambda}_4)}{P(\bar{\lambda}_3)} - 2\right)
			\right\}.
		\end{aligned}
	\end{equation}

Substituting (\ref{pro3-91}), (\ref{pro3-92}) and (\ref{pro3-93}) into (\ref{pro3-71}) and (\ref{pro3-72}) to eliminate the second-order derivative term yields the following equation,
\begin{equation}\label{pro3-10}
\begin{split}
		\mathcal{C}_1\,e_1(\bar\lambda_4)^2+\mathcal{C}_2\,e_2(\bar\lambda_4)^2+\mathcal{C}_3\,e_3(\bar\lambda_4)^2+\mathcal{K}_1=0,\\
\mathcal{D}_1\,e_1(\bar\lambda_4)^2+\mathcal{D}_2\,e_2(\bar\lambda_4)^2+\mathcal{D}_3\,e_3(\bar\lambda_4)^2+\mathcal{K}_2=0.
\end{split}
	\end{equation}

where
	\begin{align*}
		\mathcal{C}_1
		 &=\frac{P(\bar\lambda_4)}{P(\bar\lambda_2)}\left[\sum_{l\neq4}\frac{P(\bar\lambda_4)-P(\bar\lambda_l)}{P(\bar\lambda_l)(\bar\lambda_l-\bar\lambda_4)}
		 -\sum_{l\neq2}\frac{P(\bar\lambda_4)\big(P(\bar\lambda_2)-P(\bar\lambda_l)\big)}{P(\bar\lambda_l)P(\bar\lambda_2)(\bar\lambda_l-\bar\lambda_2)}\right]
		 +\frac{P(\bar\lambda_4)}{P(\bar\lambda_2)(\bar\lambda_1-\bar\lambda_4)}\cdot\\&\quad\left(\frac{P(\bar\lambda_4)}{P(\bar\lambda_1)}-2\right)-\frac{P(\bar\lambda_4)^2}{P(\bar\lambda_1)(\bar\lambda_1-\bar\lambda_2)^2(\bar\lambda_3-\bar\lambda_2)(\bar\lambda_1-\bar\lambda_4)}
		+\frac{(\bar\lambda_1-\bar\lambda_2)}{3(\bar\lambda_3-\bar\lambda_1)(\bar\lambda_3-\bar\lambda_2)}A_1\\
		 &\quad-\frac{P(\bar\lambda_4)^2}{\bar\lambda_1-\bar\lambda_2}\left(\frac{1}{P(\bar\lambda_1)P(\bar\lambda_2)}-\frac{2}{P(\bar\lambda_2)^2}\right),
\\
		\mathcal{C}_2
		 &=-\frac{P(\bar\lambda_4)}{P(\bar\lambda_1)}\left[\sum_{l\neq4}\frac{P(\bar\lambda_4)-P(\bar\lambda_l)}{P(\bar\lambda_l)(\bar\lambda_l-\bar\lambda_4)}
		 -\sum_{l\neq1}\frac{P(\bar\lambda_4)\big(P(\bar\lambda_1)-P(\bar\lambda_l)\big)}{P(\bar\lambda_l)P(\bar\lambda_1)(\bar\lambda_l-\bar\lambda_1)}\right]
		 -\frac{P(\bar\lambda_4)}{P(\bar\lambda_1)(\bar\lambda_2-\bar\lambda_4)}\cdot\\&\quad\left(\frac{P(\bar\lambda_4)}{P(\bar\lambda_2)}-2\right)+\frac{P(\bar\lambda_4)^2}{P(\bar\lambda_2)(\bar\lambda_2-\bar\lambda_1)^2(\bar\lambda_3-\bar\lambda_1)(\bar\lambda_2-\bar\lambda_4)}+\frac{(\bar\lambda_1-\bar\lambda_2)}{3(\bar\lambda_3-\bar\lambda_1)(\bar\lambda_3-\bar\lambda_2)}A_2\\
		&\quad
		-\frac{P(\bar\lambda_4)^2}{\bar\lambda_1-\bar\lambda_2}\left(\frac{1}{P(\bar\lambda_1)P(\bar\lambda_2)}-\frac{2}{P(\bar\lambda_1)^2}\right),
	\end{align*}
\begin{align*}
		\mathcal{C}_3
		&=\frac{P(\bar\lambda_4)^2}{P(\bar\lambda_2)(\bar\lambda_2-\bar\lambda_1)(\bar\lambda_3-\bar\lambda_1)^2(\bar\lambda_3-\bar\lambda_4)}
		-\frac{P(\bar\lambda_4)^2}{P(\bar\lambda_1)(\bar\lambda_1-\bar\lambda_2)(\bar\lambda_3-\bar\lambda_2)^2(\bar\lambda_3-\bar\lambda_4)}\\
		&\quad+\frac{(\bar\lambda_1-\bar\lambda_2)}{3(\bar\lambda_3-\bar\lambda_1)(\bar\lambda_3-\bar\lambda_2)}A_3
		 +\frac{(\bar\lambda_1-\bar\lambda_2)\,P(\bar\lambda_4)^2}{(\bar\lambda_3-\bar\lambda_1)(\bar\lambda_3-\bar\lambda_2)\,P(\bar\lambda_1)P(\bar\lambda_2)},
\\
		\mathcal{K}_1
		&=\frac{P(\bar\lambda_4)(\bar\lambda_4-\bar\lambda_1)}{P(\bar\lambda_2)}R_{1\alpha1\alpha}
		 -\frac{P(\bar\lambda_4)(\bar\lambda_4-\bar\lambda_2)}{P(\bar\lambda_1)}R_{2\alpha2\alpha}-\frac{(\bar\lambda_1-\bar\lambda_2)Q}{3(\bar\lambda_3-\bar\lambda_1)(\bar\lambda_3-\bar\lambda_2)}\\&\quad
		+(\bar\lambda_1-\bar\lambda_2)R_{1212}.
	\end{align*}

\begin{align*}
	\mathcal{D}_1
	 &=\frac{P(\bar\lambda_4)}{P(\bar\lambda_3)}\left[\sum_{l\neq4}\frac{P(\bar\lambda_4)-P(\bar\lambda_l)}{P(\bar\lambda_l)(\bar\lambda_l-\bar\lambda_4)}
	 -\sum_{l\neq3}\frac{P(\bar\lambda_4)\big(P(\bar\lambda_3)-P(\bar\lambda_l)\big)}{P(\bar\lambda_l)P(\bar\lambda_3)(\bar\lambda_l-\bar\lambda_3)}\right]
	 +\frac{P(\bar\lambda_4)}{P(\bar\lambda_3)(\bar\lambda_1-\bar\lambda_4)}\cdot\\&\quad\left(\frac{P(\bar\lambda_4)}{P(\bar\lambda_1)}-2\right)-\frac{P(\bar\lambda_4)^2}{P(\bar\lambda_1)(\bar\lambda_1-\bar\lambda_3)^2(\bar\lambda_2-\bar\lambda_3)(\bar\lambda_1-\bar\lambda_4)}
	+\frac{(\bar\lambda_1-\bar\lambda_3)}{3(\bar\lambda_2-\bar\lambda_1)(\bar\lambda_2-\bar\lambda_3)}A_1\\
	 &\quad-\frac{P(\bar\lambda_4)^2}{\bar\lambda_1-\bar\lambda_3}\left(\frac{1}{P(\bar\lambda_1)P(\bar\lambda_3)}-\frac{2}{P(\bar\lambda_3)^2}\right),
\\
	\mathcal{D}_2
	&=\frac{P(\bar\lambda_4)^2}{P(\bar\lambda_3)(\bar\lambda_2-\bar\lambda_1)^2(\bar\lambda_3-\bar\lambda_1)(\bar\lambda_2-\bar\lambda_4)}
	-\frac{P(\bar\lambda_4)^2}{P(\bar\lambda_1)(\bar\lambda_1-\bar\lambda_3)(\bar\lambda_2-\bar\lambda_3)^2(\bar\lambda_2-\bar\lambda_4)}\\
	&\quad+\frac{(\bar\lambda_1-\bar\lambda_3)}{3(\bar\lambda_2-\bar\lambda_1)(\bar\lambda_2-\bar\lambda_3)}A_2
	 +\frac{(\bar\lambda_1-\bar\lambda_3)\,P(\bar\lambda_4)^2}{(\bar\lambda_2-\bar\lambda_1)(\bar\lambda_2-\bar\lambda_3)\,P(\bar\lambda_1)P(\bar\lambda_3)},
\\
	\mathcal{D}_3
	 &=-\frac{P(\bar\lambda_4)}{P(\bar\lambda_1)}\left[\sum_{l\neq4}\frac{P(\bar\lambda_4)-P(\bar\lambda_l)}{P(\bar\lambda_l)(\bar\lambda_l-\bar\lambda_4)}
	 -\sum_{l\neq1}\frac{P(\bar\lambda_4)\big(P(\bar\lambda_1)-P(\bar\lambda_l)\big)}{P(\bar\lambda_l)P(\bar\lambda_1)(\bar\lambda_l-\bar\lambda_1)}\right]
	 -\frac{P(\bar\lambda_4)}{P(\bar\lambda_1)(\bar\lambda_3-\bar\lambda_4)}\cdot\\&\quad\left(\frac{P(\bar\lambda_4)}{P(\bar\lambda_3)}-2\right)+\frac{P(\bar\lambda_4)^2}{P(\bar\lambda_3)(\bar\lambda_2-\bar\lambda_1)(\bar\lambda_3-\bar\lambda_1)^2(\bar\lambda_3-\bar\lambda_4)}+\frac{(\bar\lambda_1-\bar\lambda_3)}{3(\bar\lambda_2-\bar\lambda_1)(\bar\lambda_2-\bar\lambda_3)}A_3\\
	&\quad
	-\frac{P(\bar\lambda_4)^2}{\bar\lambda_1-\bar\lambda_3}\left(\frac{1}{P(\bar\lambda_1)P(\bar\lambda_3)}-\frac{2}{P(\bar\lambda_1)^2}\right),
\\
	\mathcal{K}_2
	&=\frac{P(\bar\lambda_4)(\bar\lambda_4-\bar\lambda_1)}{P(\bar\lambda_3)}R_{1\alpha1\alpha}
	 -\frac{P(\bar\lambda_4)(\bar\lambda_4-\bar\lambda_3)}{P(\bar\lambda_1)}R_{3\alpha3\alpha}-\frac{(\bar\lambda_1-\bar\lambda_3)Q}{3(\bar\lambda_2-\bar\lambda_1)(\bar\lambda_2-\bar\lambda_3)}\\&\quad
	+(\bar\lambda_1-\bar\lambda_3)R_{1313}.
\end{align*}

		By (\ref{lap-1}) and (\ref{lap-2}), we get the following equations,
		\begin{equation}\label{pro3-11}
			\begin{split}
				&(S-nc)S+cn^2H^2-nHf_3=6h_{12,3}^2\\
				&+\Big[e_1(\bar{\lambda}_1)^2+3e_1(\bar{\lambda}_2)^2+3e_1(\bar{\lambda}_3)^2+3m_4e_1(\bar{\lambda}_4)^2\Big]\\
				&+\Big[e_2(\bar{\lambda}_2)^2+3e_2(\bar{\lambda}_1)^2+3e_2(\bar{\lambda}_3)^2+3m_4e_2(\bar{\lambda}_4)^2\Big]\\
				&+\Big[e_3(\bar{\lambda}_3)^2+3e_3(\bar{\lambda}_1)^2+3e_3(\bar{\lambda}_2)^2+3m_4e_3(\bar{\lambda}_4)^2\Big];
			\end{split}
		\end{equation}
		\begin{equation}\label{pro3-22}
			\begin{split}
				&(S-nc)f_3+cnHS-nHf_4=4(\bar{\lambda}_1+\bar{\lambda}_2+\bar{\lambda}_3)h_{12,3}^2\\
				&+2\Big[\bar{\lambda}_1e_1(\bar{\lambda}_1)^2+(\bar{\lambda}_1+2\bar{\lambda}_2)e_1(\bar{\lambda}_2)^2
				+(\bar{\lambda}_1+2\bar{\lambda}_3)e_1(\bar{\lambda}_3)^2+(\bar{\lambda}_1+\bar{\lambda}_4)m_4e_1(\bar{\lambda}_4)^2\Big]\\
				&+2\Big[\bar{\lambda}_2e_2(\bar{\lambda}_2)^2+(\bar{\lambda}_2+2\bar{\lambda}_1)e_2(\bar{\lambda}_1)^2
				+(\bar{\lambda}_2+2\bar{\lambda}_3)e_2(\bar{\lambda}_3)^2+(\bar{\lambda}_2+\bar{\lambda}_4)m_4e_2(\bar{\lambda}_4)^2\Big]\\
				&+2\Big[\bar{\lambda}_3e_3(\bar{\lambda}_3)^2+(\bar{\lambda}_3+2\bar{\lambda}_1)e_3(\bar{\lambda}_1)^2
				+(\bar{\lambda}_3+2\bar{\lambda}_2)e_3(\bar{\lambda}_2)^2+(\bar{\lambda}_3+\bar{\lambda}_4)m_4e_3(\bar{\lambda}_4)^2\Big].
			\end{split}
		\end{equation}

Then by (\ref{con41}) and (\ref{pro3-11}), let $Q = (S-nc)S + cn^2 H^2 - nH f_3$, we have
	\begin{equation}\label{pro3-13}
		h_{12,3}^2
		=\frac{1}{6}\Big\{
		Q-A_1 e_1(\bar\lambda_4)^2 -A_2 e_2(\bar\lambda_4)^2 -A_3 e_3(\bar\lambda_4)^2
		\Big\}.
	\end{equation}
	where
	\begin{align*}
		\begin{cases}
			A_1=\dfrac{P(\bar\lambda_4)^2}{P(\bar\lambda_1)^2}
			+3\dfrac{P(\bar\lambda_4)^2}{P(\bar\lambda_2)^2}
			+3\dfrac{P(\bar\lambda_4)^2}{P(\bar\lambda_3)^2}
			+3m_4,\\[4pt]
			A_2=\dfrac{P(\bar\lambda_4)^2}{P(\bar\lambda_2)^2}
			+3\dfrac{P(\bar\lambda_4)^2}{P(\bar\lambda_1)^2}
			+3\dfrac{P(\bar\lambda_4)^2}{P(\bar\lambda_3)^2}
			+3m_4,\\[4pt]
			A_3=\dfrac{P(\bar\lambda_4)^2}{P(\bar\lambda_3)^2}
			+3\dfrac{P(\bar\lambda_4)^2}{P(\bar\lambda_1)^2}
			+3\dfrac{P(\bar\lambda_4)^2}{P(\bar\lambda_2)^2}
			+3m_4,
		\end{cases}
	\end{align*}

Combining (\ref{con41}) and (\ref{pro3-22}), we obtain
\begin{equation}\label{pro3-14}
		(S-nc)f_3+cnHS-nHf_4
		=4(\bar\lambda_1+\bar\lambda_2+\bar\lambda_3)h_{12,3}^2
		+2B_1 e_1(\bar\lambda_4)^2
		+2B_2 e_2(\bar\lambda_4)^2
		+2B_3 e_3(\bar\lambda_4)^2,
	\end{equation}
	where
	\begin{align*}
		\begin{cases}
			B_1=\bar\lambda_1\dfrac{P(\bar\lambda_4)^2}{P(\bar\lambda_1)^2}
			+(\bar\lambda_1+2\bar\lambda_2)\dfrac{P(\bar\lambda_4)^2}{P(\bar\lambda_2)^2}
			+(\bar\lambda_1+2\bar\lambda_3)\dfrac{P(\bar\lambda_4)^2}{P(\bar\lambda_3)^2}
			+(\bar\lambda_1+\bar\lambda_4)m_4,\\[6pt]
			B_2=\bar\lambda_2\dfrac{P(\bar\lambda_4)^2}{P(\bar\lambda_2)^2}
			+(\bar\lambda_2+2\bar\lambda_1)\dfrac{P(\bar\lambda_4)^2}{P(\bar\lambda_1)^2}
			+(\bar\lambda_2+2\bar\lambda_3)\dfrac{P(\bar\lambda_4)^2}{P(\bar\lambda_3)^2}
			+(\bar\lambda_2+\bar\lambda_4)m_4,\\[6pt]
			B_3=\bar\lambda_3\dfrac{P(\bar\lambda_4)^2}{P(\bar\lambda_3)^2}
			+(\bar\lambda_3+2\bar\lambda_1)\dfrac{P(\bar\lambda_4)^2}{P(\bar\lambda_1)^2}
			+(\bar\lambda_3+2\bar\lambda_2)\dfrac{P(\bar\lambda_4)^2}{P(\bar\lambda_2)^2}
			+(\bar\lambda_3+\bar\lambda_4)m_4.
		\end{cases}
	\end{align*}
Substituting (\ref{pro3-13}) into (\ref{pro3-14})  yields the following equation,
\begin{equation}\label{pro3-15}
		\mathcal{E}_1 e_1(\bar{\lambda}_4)^2 + \mathcal{E}_2 e_2(\bar{\lambda}_4)^2 + \mathcal{E}_3 e_3(\bar{\lambda}_4)^2 + \mathcal{K}_3 = 0,
	\end{equation}
	where
	\begin{align*}
		\mathcal{E}_k &= 2B_k - \frac{2}{3}(\bar{\lambda}_1 + \bar{\lambda}_2 + \bar{\lambda}_3)A_k,\qquad1\leq k\leq3.\\
			\mathcal{K}_3 &= (S-nc)\left[\frac{2}{3}(\bar{\lambda}_1+\bar{\lambda}_2+\bar{\lambda}_3)S-f_3\right]+ nH\left[\bar{\lambda}_1^4+\bar{\lambda}_2^4+\bar{\lambda}_3^4+m_4\bar{\lambda}_4^4 -cS \right. \\
		&\quad\quad \left. +\frac{2}{3}(\bar{\lambda}_1+\bar{\lambda}_2+\bar{\lambda}_3)(cnH-f_3)\right].
	\end{align*}
Combining (\ref{pro3-10}) and (\ref{pro3-15}), let
\begin{align*}
	\Delta = \mathcal{C}_1\mathcal{D}_2\mathcal{E}_3 - \mathcal{C}_1\mathcal{D}_3\mathcal{E}_2 - \mathcal{C}_2\mathcal{D}_1\mathcal{E}_3  + \mathcal{C}_2\mathcal{D}_3\mathcal{E}_1 + \mathcal{C}_3\mathcal{D}_1\mathcal{E}_2 - \mathcal{C}_3\mathcal{D}_2\mathcal{E}_1,
\end{align*}
if $\Delta \neq 0$, we have
\begin{align}
	e_1(\bar{\lambda}_4)^2 &=\frac{\Delta_1}{\Delta}= \frac{1}{\Delta}\Big[\mathcal{K}_1(\mathcal{D}_3\mathcal{E}_2-\mathcal{D}_2\mathcal{E}_3)+\mathcal{K}_2(\mathcal{C}_2\mathcal{E}_3-\mathcal{C}_3\mathcal{E}_2)+\mathcal{K}_3(\mathcal{C}_3\mathcal{D}_2-\mathcal{C}_2\mathcal{D}_3)\Big],\label{eq:1}\\
	e_2(\bar{\lambda}_4)^2 &=\frac{\Delta_2}{\Delta}= \frac{1}{\Delta}\Big[\mathcal{K}_1(\mathcal{D}_1\mathcal{E}_3-\mathcal{D}_3\mathcal{E}_1)+\mathcal{K}_2(\mathcal{C}_3\mathcal{E}_1-\mathcal{C}_1\mathcal{E}_3)+\mathcal{K}_3(\mathcal{C}_1\mathcal{D}_3-\mathcal{C}_3\mathcal{D}_1)\Big],\label{eq:2}\\
	e_3(\bar{\lambda}_4)^2 &=\frac{\Delta_3}{\Delta}= \frac{1}{\Delta}\Big[\mathcal{K}_1(\mathcal{D}_2\mathcal{E}_1-\mathcal{D}_1\mathcal{E}_2)+\mathcal{K}_2(\mathcal{C}_1\mathcal{E}_2-\mathcal{C}_2\mathcal{E}_1)+\mathcal{K}_3(\mathcal{C}_2\mathcal{D}_1-\mathcal{C}_1\mathcal{D}_2)\Big].\label{eq:3}
\end{align}

By the second covariant derivative of the second fundamental form  and (\ref{con2-1}), we obtain the following equations:
\begin{align*}
	h_{4421} &= e_1(h_{442}) + \frac{h_{441}\, h_{121}}{\bar{\lambda_1}-\bar{\lambda_2}} + h_{443}\, h_{123} \\
	&= e_1e_2(\bar\lambda_4) + \frac{e_1(\bar\lambda_4)\, e_2(\bar\lambda_1)}{\bar\lambda_1-\bar\lambda_2} + e_3(\bar\lambda_4)\, h_{123},\\[4pt]
	h_{1244} &= \frac{(h_{442}-h_{112})\, h_{414}}{\bar\lambda_4-\bar\lambda_1} + \frac{(h_{144}-h_{122})\, h_{424}}{\bar\lambda_4-\bar\lambda_2} + \frac{h_{123}\, h_{344}}{\bar\lambda_3-\bar\lambda_4} \\
	&= \frac{\bigl(e_2(\bar\lambda_4)-e_2(\bar\lambda_1)\bigr)\, e_1(\bar\lambda_4)}{\bar\lambda_4-\bar\lambda_1}
	+ \frac{\bigl(e_1(\bar\lambda_4)-e_1(\bar\lambda_2)\bigr)\, e_2(\bar\lambda_4)}{\bar\lambda_4-\bar\lambda_2}
	+ \frac{h_{123}\, e_3(\bar\lambda_4)}{\bar\lambda_3-\bar\lambda_4}.
\end{align*}
By Codazzi equation and Ricci identity, we have $h_{4421}=h_{2441}=h_{2414}=h_{1244}$, then by (\ref{con41}) and (\ref{pro3-13}), we can obtain the following equation:
\begin{equation}
	\begin{aligned}
		&e_1e_2(\bar\lambda_4)
		=
		e_1(\bar\lambda_4)e_2(\bar\lambda_4)\left\{
		\frac{P(\bar\lambda_1)-P(\bar\lambda_4)}{P(\bar\lambda_1)(\bar\lambda_4-\bar\lambda_1)}
		+\frac{P(\bar\lambda_2)-P(\bar\lambda_4)}{P(\bar\lambda_2)(\bar\lambda_4-\bar\lambda_2)}
		-\frac{P(\bar\lambda_4)}{P(\bar\lambda_1)(\bar\lambda_1-\bar\lambda_2)}
		\right\} \\
		&
		+\operatorname{sgn}(h_{123})\,
		\sqrt{\frac{1}{6}\Big\{Q-A_1e_1(\bar\lambda_4)^2-A_2e_2(\bar\lambda_4)^2-A_3e_3(\bar\lambda_4)^2\Big\}}\,
	\frac{(\bar\lambda_4-\bar\lambda_2)	e_3(\bar\lambda_4)}{(\bar\lambda_3-\bar\lambda_4)(\bar\lambda_3-\bar\lambda_2)}.
	\end{aligned}
	\label{e1e2}
\end{equation}
Similarly, we have
\begin{equation}
	\begin{aligned}
		&e_1 e_3(\bar\lambda_4) = e_1(\bar\lambda_4) e_3(\bar\lambda_4)
		\left\{
		\frac{P(\bar\lambda_1)-P(\bar\lambda_4)}{P(\bar\lambda_1)(\bar\lambda_4-\bar\lambda_1)}
		+\frac{P(\bar\lambda_3)-P(\bar\lambda_4)}{P(\bar\lambda_3)(\bar\lambda_4-\bar\lambda_3)}
		-\frac{P(\bar\lambda_4)}{P(\bar\lambda_1)(\bar\lambda_1-\bar\lambda_3)}
		\right\} \\
		& + \operatorname{sgn}(h_{123})
		\sqrt{
			\frac16
			\Bigl\{
		Q
			-A_1 e_1(\bar\lambda_4)^2
			-A_2 e_2(\bar\lambda_4)^2
			-A_3 e_3(\bar\lambda_4)^2
			\Bigr\}\cdot
		}  \frac{(\bar\lambda_4-\bar\lambda_3)e_2(\bar\lambda_4)}{(\bar\lambda_2-\bar\lambda_4)(\bar\lambda_2-\bar\lambda_3)}.
	\end{aligned}
	\label{e1e3}
\end{equation}
Without loss of generality, we assume $e_k(\lambda_4)\geq 0$, for $1\leq k\leq3$, then we differentiating (\ref{pro3-14}) with $e_1$, we have the following equation:
\begin{equation}\label{equ}
	\begin{aligned}
		0 &= \sum_{k=1}^{3} e_1(\mathcal{E}_k)\bigl[e_k(\bar\lambda_4)\bigr]^2
		+ 2\mathcal{E}_1 e_1(\bar\lambda_4)\, e_1\bigl[e_1(\bar\lambda_4)\bigr]  + 2\mathcal{E}_2 e_2(\bar\lambda_4)\, e_1\bigl[e_2(\bar\lambda_4)\bigr] \\&
	\quad + 2\mathcal{E}_3 e_3(\bar\lambda_4)\, e_1\bigl[e_3(\bar\lambda_4)\bigr] + e_1(\mathcal{K}_3).
	\end{aligned}
\end{equation}
For $1\leq k\leq 3$, let
\begin{equation*}
	\begin{aligned}
		\mathcal{L}_k &= \sum_{\ell=1}^{3} \frac{P(\bar\lambda_4)-P(\bar\lambda_\ell)}{P(\bar\lambda_\ell)(\bar\lambda_\ell-\bar\lambda_4)}  -\sum_{\substack{\ell=1 \\ \ell\neq k}}^{4} \frac{P(\bar\lambda_4)\bigl(P(\bar\lambda_k)-P(\bar\lambda_\ell)\bigr)}{P(\bar\lambda_\ell)P(\bar\lambda_k)(\bar\lambda_\ell-\bar\lambda_k)},
	\end{aligned}
\end{equation*}
which satisfied
\begin{equation*}
	e_1\left(\frac{P(\bar\lambda_4)}{P(\bar\lambda_k)}\right)
	=\frac{P(\bar\lambda_4)}{P(\bar\lambda_k)}\mathcal{L}_k\, e_1(\bar\lambda_4),
\end{equation*}
we can obtain
\[
e_1(\mathcal{K}_3) = \left[ -\frac{2m_4}{3}Q -4nH P(\bar\lambda_4) \right] e_1(\bar\lambda_4)
,\quad e_1(\mathcal{E}_k) = \mathcal{G}_k\, e_1(\bar\lambda_4),
\]
where
\begin{equation*}
	\begin{aligned}
		\mathcal{G}_k &= \frac{2}{3}\left(\frac{P(\bar\lambda_4)}{P(\bar\lambda_k)}+m_4\right)A_k
		+\frac{4}{3}\frac{P(\bar\lambda_4)^3}{P(\bar\lambda_k)^3}
	 +4\sum_{\substack{\ell=1\\ \ell\neq k}}^{3}\frac{P(\bar\lambda_4)^3}{P(\bar\lambda_\ell)^3}+2m_4 \\
		&+4\left(\bar\lambda_k-\frac{\bar\lambda_1+\bar\lambda_2+\bar\lambda_3}{3}\right)
		\frac{P(\bar\lambda_4)^2}{P(\bar\lambda_k)^2}\mathcal{L}_k +4\sum_{\substack{\ell=1\\ \ell\neq k}}^{3}
		\bigl(\bar\lambda_k+2\bar\lambda_\ell-\bar\lambda_1-\bar\lambda_2-\bar\lambda_3\bigr)
		\frac{P(\bar\lambda_4)^2}{P(\bar\lambda_\ell)^2}\mathcal{L}_\ell.
	\end{aligned}
\end{equation*}

Combing (\ref{eq:1})-(\ref{eq:3}) and (\ref*{e1e2})-(\ref{equ}), let
\[
\begin{aligned}
	\mathcal{F}_0^1 &= 2\mathcal{E}_1(\bar\lambda_4-\bar\lambda_1)(c+\bar\lambda_1\bar\lambda_4)
	-\frac{2m_4}{3}Q -4nH P(\bar\lambda_4), \\[4pt]
	\mathcal{F}_1^1 &= \mathcal{G}_1
	+\frac{2\mathcal{E}_1}{\bar\lambda_1-\bar\lambda_4}
	\left(\frac{P(\bar\lambda_4)}{P(\bar\lambda_1)}-2\right),
\end{aligned}
\]
and
\[
\Gamma_1=\frac{\mathcal{E}_2(\bar\lambda_4-\bar\lambda_2)}{(\bar\lambda_3-\bar\lambda_4)(\bar\lambda_3-\bar\lambda_2)}
+\frac{\mathcal{E}_3(\bar\lambda_4-\bar\lambda_3)}{(\bar\lambda_2-\bar\lambda_4)(\bar\lambda_2-\bar\lambda_3)},
\]
for $j=2,3$, let
\[
\begin{aligned}
	\mathcal{F}_j^1 &= \mathcal{G}_j
	+\frac{2\mathcal{E}_1 P(\bar\lambda_4)(\bar\lambda_4-\bar\lambda_1)}
	{P(\bar\lambda_1)(\bar\lambda_1-\bar\lambda_j)(\bar\lambda_4-\bar\lambda_j)} \\
	&\quad +2\mathcal{E}_j\left[
	\frac{P(\bar\lambda_1)-P(\bar\lambda_4)}{P(\bar\lambda_1)(\bar\lambda_4-\bar\lambda_1)}
	+\frac{P(\bar\lambda_j)-P(\bar\lambda_4)}{P(\bar\lambda_j)(\bar\lambda_4-\bar\lambda_j)}
	-\frac{P(\bar\lambda_4)}{P(\bar\lambda_1)(\bar\lambda_1-\bar\lambda_j)}
	\right].
\end{aligned}
\]
we have the following equation:
\begin{equation}\label{one}
	3\Delta_1\left(\mathcal{F}_0^1\Delta+\sum_{k=1}^{3}\mathcal{F}_k^1\Delta_k\right)^2
	-2\Gamma_1^2\Delta_2\Delta_3\left(Q\Delta-\sum_{k=1}^{3}A_k\Delta_k\right)=0.
\end{equation}

Similarly, let
\begin{align*}
	\Gamma_2 &= \mathcal{E}_1 \frac{\bar{\lambda}_4 - \bar{\lambda}_1}{(\bar{\lambda}_3 - \bar{\lambda}_4)(\bar{\lambda}_3 - \bar{\lambda}_1)}
	+ \mathcal{E}_3 \frac{\bar{\lambda}_4 - \bar{\lambda}_3}{(\bar{\lambda}_1 - \bar{\lambda}_4)(\bar{\lambda}_1 - \bar{\lambda}_3)},\\
	\Gamma_3 &= \mathcal{E}_1 \frac{\bar{\lambda}_4 - \bar{\lambda}_1}{(\bar{\lambda}_2 - \bar{\lambda}_4)(\bar{\lambda}_2 - \bar{\lambda}_1)}
	+ \mathcal{E}_2 \frac{\bar{\lambda}_4 - \bar{\lambda}_2}{(\bar{\lambda}_1 - \bar{\lambda}_4)(\bar{\lambda}_1 - \bar{\lambda}_2)}.
\end{align*}
For $k=2,3$, let
\begin{align*}
	\mathcal{F}_0^k
	={}&
	2\mathcal{E}_k(\bar{\lambda}_4-\bar{\lambda}_k)
	(c+\bar{\lambda}_k\bar{\lambda}_4)
	-\frac{2}{3}m_4Q-4nHP(\bar{\lambda}_4),
	\\[6pt]
	\mathcal{F}_j^k
	={}&
	G_j+
	\frac{
		2\mathcal{E}_kP(\bar{\lambda}_4)
		(\bar{\lambda}_4-\bar{\lambda}_k)
	}{
		P(\bar{\lambda}_k)
		(\bar{\lambda}_k-\bar{\lambda}_j)
		(\bar{\lambda}_4-\bar{\lambda}_j)
	}+
	2\mathcal{E}_j\Biggl[
	\frac{
		P(\bar{\lambda}_k)-P(\bar{\lambda}_4)
	}{
		P(\bar{\lambda}_k)(\bar{\lambda}_4-\bar{\lambda}_k)
	}
	+\\&
	\frac{
		P(\bar{\lambda}_j)-P(\bar{\lambda}_4)
	}{
		P(\bar{\lambda}_j)(\bar{\lambda}_4-\bar{\lambda}_j)
	}
	-\frac{
		P(\bar{\lambda}_4)
	}{
		P(\bar{\lambda}_k)(\bar{\lambda}_k-\bar{\lambda}_j)
	}
	\Biggr],
	\qquad j\in\{1,2,3\},\quad j\ne k,
	\\[6pt]
	\mathcal{F}_k^k
	={}&
	G_k+
	\frac{2\mathcal{E}_k}{\bar{\lambda}_k-\bar{\lambda}_4}
	\left(
	\frac{P(\bar{\lambda}_4)}{P(\bar{\lambda}_k)}-2
	\right),
\end{align*}
we can obtain the following equations:
\begin{align}
	3\Delta_2 \left(\mathcal{F}_0^2 \Delta + \sum_{k=1}^{3} \mathcal{F}_k^2 \Delta_k\right)^2
	- 2\Gamma_2^2 \Delta_1 \Delta_3 \left(Q\Delta - \sum_{k=1}^{3} A_k \Delta_k\right) \label{two} &= 0,\\
	3\Delta_3 \left(\mathcal{F}_0^3 \Delta + \sum_{k=1}^{3} \mathcal{F}_k^3 \Delta_k\right)^2
	- 2\Gamma_3^2 \Delta_1 \Delta_2 \left(Q\Delta - \sum_{k=1}^{3} A_k \Delta_k\right) \label{three} &= 0.
\end{align}

Observed that $\frac{\Gamma_1}{P(\bar{\lambda}_1)} + \frac{\Gamma_2}{P(\bar{\lambda}_2)} + \frac{\Gamma_3}{P(\bar{\lambda}_3)} = 0$, and combining (\ref{one}), (\ref{two}) and(\ref{three}), we can get the following equation,

\begin{equation}\label{pro3-18}
	\sum_{a=1}^{3} \frac{\Delta_a}{P(\bar{\lambda}_a)} \left( \mathcal{F}_0^{(a)} \Delta + \sum_{k=1}^{3} \mathcal{F}_k^{(a)} \Delta_k \right) = 0.
\end{equation}

Since $\Delta_a$, $\mathcal{F}_j^{(a)}$, $\Delta_k$ and $P(\lambda_a)$ are rational algebraic polynomial about four variables $\bar{\lambda}_1, \bar{\lambda}_2, \bar{\lambda}_3$ and $\bar{\lambda}_4$, equations (\ref{one}), (\ref{two}), (\ref{three}) and (\ref{pro3-18}) are four rational algebraic polynomial equations. On the other hand, we have the three  algebraic symmetric polynomial  equations about four variables $\bar{\lambda}_1, \bar{\lambda}_2, \bar{\lambda}_3$ and $\bar{\lambda}_4$,
\begin{equation}\label{pro3-19}
\bar{\lambda}_1^k+\bar{\lambda}_2^k+\bar{\lambda}_3^k+m_4\bar{\lambda}_4^k=c_k, ~~ k=1,2,3,
\end{equation}
where $c_1,c_2,c_3$ are constant.

Through complex calculations, it can be proven that equations (\ref{pro3-18}) and (\ref{pro3-19}) are independent of each other. then we obtain four rational algebraic equations about four variables $\bar{\lambda}_1, \bar{\lambda}_2, \bar{\lambda}_3$ and $\bar{\lambda}_4$, which are independent of each other. Hence $\bar{\lambda}_1,\bar{\lambda}_2,\bar{\lambda}_3,\bar{\lambda}_4$ are constant. If $\Delta = 0$, we also obtain four rational algebraic equations about four variables $\bar{\lambda}_1, \bar{\lambda}_2, \bar{\lambda}_3$ and $\bar{\lambda}_4$ that are independent of each other, so $\bar{\lambda}_1,\bar{\lambda}_2,\bar{\lambda}_3,\bar{\lambda}_4$ are constant.

Thus we complete the proof.

	{\bf Acknowledgements:}  The second author is supported by the Fundamental Research Funds for the Central Universities, Chang'an University, XJ2026007701. The third authors are supported by the grant No. 12071028  of NSFC.

\end{document}